\documentclass[11pt, article]{amsart}

\usepackage{graphicx, epstopdf}
\usepackage{cases}

\usepackage{enumerate}
\usepackage{bbm}
\usepackage{amssymb}
\usepackage{amscd}
\usepackage{amsfonts,latexsym,amsmath,amsxtra,mathdots,amssymb,latexsym,mathtools}
\usepackage{mathabx}
\usepackage[all,cmtip]{xy}

\usepackage{colordvi}
\usepackage{multicol}
\usepackage{hyperref,dirtytalk}
\usepackage{tikz-cd}
\usepackage{cleveref}
\usepackage{float}
\usepackage{setspace, floatflt}
\usepackage{tensor}
\usepackage[normalem]{ulem}

\allowdisplaybreaks

 \newcommand{\BC}{{\mathbb {C}}} 
  \newcommand{\BH}{{\mathbb {H}}}
  
 \newcommand{\BN}{{\mathbb {N}}} 
 \newcommand{\BR}{{\mathbb {R}}} 
  
 \newcommand{\BZ}{{\mathbb {Z}}}

\newcommand{\CM}{{\mathcal {M}}}

\newcommand{\GL}{{\mathrm {GL}}} 
\newcommand{\SL}{{\mathrm {SL}}}

\newcommand{\U}{{\mathrm{U}}}

 \newcommand{\N}{\mathrm{N}}

\newcommand{\re}{\mathfrak{R}}

\newcommand{\Hom}{\mathrm{Hom}}
  
\newcommand{\Ad}{\mathrm{Ad}}
\newcommand{\triv}{{\mathbf{1}}}

\newcommand{\us}{\underline{s}}

\def\-{^{-1}}

\def\1{\mathbf{1}}

\def\diag{\mathrm{diag}}

\renewcommand{\Im}{{\mathrm{Im}\,}}

\def\ima{\Im}

\makeatletter
\g@addto@macro\normalsize{\setlength\abovedisplayskip{3pt}}
\makeatother

\makeatletter
\g@addto@macro\normalsize{\setlength\belowdisplayskip{3pt}}
\makeatother

\newcommand{\delete}[1]{}

\theoremstyle{plain}

\newtheorem{thm}{Theorem}[section] 
\newtheorem{cor}[thm]{Corollary}

\newtheorem*{st*}{Statement}
\newtheorem*{q*}{Question}
\newtheorem{ass}{Assumption}

\newtheorem*{thm*}{Theorem}
\newtheorem*{lm*}{Lemma}
\newtheorem*{cor*}{Corollary}

\newtheorem{conjecture}{Conjecture}

\newtheorem*{concl*}{Conclusion}
\newtheorem*{rem*}{Remark}
\newtheorem{lem}[thm]{Lemma} 
 \newtheorem{prop}[thm]{Proposition}

\newtheorem{df}[thm]{Definition}
\newtheorem{fct}[thm]{Fact}

\newtheorem {rem}[thm]{Remark}

\numberwithin{equation}{section}

\begin{document}

	\title{On completeness of local intertwining periods} 
	\date{\today}
		\author{H. Lu}
	\address{Hengfei Lu. School of Mathematical Sciences, Beihang University, 9 Nansan Street, Shahe Higher Education Park, Changping, Beijing 102206 P.R.C}
	\email{luhengfei@buaa.edu.cn}
	
	\author{N. Matringe}
	\address{Nadir Matringe. Shanghai Institute for Mathematics and Interdisciplinary Sciences, Block A, International Innovation Plaza, No. 657 Songhu Road, Yangpu District, Shanghai, China}
	\email{nadirmatringe@outlook.fr}

	\maketitle
	
\vspace{-0.5cm}	
	
\begin{abstract} 
	In this paper we study the problem of explicitly describing the space of invariant linear forms on induced distinguished representations in terms of invariant linear forms on the inducing representation. More precisely, for certain tempered reductive symmetric pairs $(G,H)$ over a local field of characteristic zero, which we call unimodular in this paper, we study under which condition on the inducing representation, the space of $H$-invariant linear forms on a parabolically induced representation of $G$ is generated by regularized intertwining periods attached to admissible parabolic orbits in $G/H$, as defined in the work of Matringe--Offen--Yang. We conjecture that it is the case when the inducing representation is square-integrable. Under this assumption we actually conjecture that one can replace regularized by normalized intertwining periods. We then verify the conjecture on known examples, and prove it for various pairs where $G$ has semi-simple split rank one. 
\end{abstract}

\vspace{1cm} 

Let $(G,H)$ be a reductive symmetric pair over a local field of characteristic zero, and let $\pi$ be a smooth admissible representation of $G$. When $G$ is a real reductive group, by smooth admissible, we mean smooth admissible of moderate growth, and  Fr\'echet, as in \cite[Chapter 11]{RRG2}. We denote by $\Hom_H(\pi,\BC)$ the space of  $H$-invariant linear forms on $\pi$, which we assume to be moreover continuous in the archimedean case. Understanding when $\Hom_H(\pi,\BC)$ is nonzero, i.e., when $\pi$ is $H$-distinguished, and in such case determining the dimension of 
$\Hom_H(\pi,\BC)$, is a central problem in the part of the local relative Langlands program concerned with symmetric pairs. The most interesting problems are usually stated in terms of the Langlands parameter of $\pi$, and the distinction of $\pi$ as well as the dimension of $\Hom_H(\pi,\BC)$ when $\pi$ is distinguished are also predicted in such terms. We refer for example to \cite{Prel}, \cite{SV}, \cite{Wan}, \cite{GGP} or \cite{BSV} for such predictions. In order to prove conjectures of this type, usually one first treats the case when $\pi$ is cuspidal or square-integrable, most of the time using relative trace formula methods as for example in \cite{WGGP}, \cite{MWGGP}, \cite{FLO}, \cite{BPGalsq}. Inspired by the global works of \cite{JLR} and \cite{LR}, and building on the local works of \cite{BD}, \cite{Carmona-Delorme-H-inv-form-FE-JFA}, \cite{Brylinski-Delorme-H-inv-form-MeroExtension-Invention}, an efficient tool, called local intertwining periods, was introduced in \cite{MatJFA} for certain Galois pairs attached to inner forms of the general linear group, in order to reduce the study of distinction on induced representations to the case of square-integrable or cuspidal representations. A much more general version of local intertwining periods, for pairs $(G,H)$ that we call unimodular in this paper, was then provided by \cite{MO}, and the results of \cite{MO} were further extended in \cite{MOY}.  
Intertwining periods are meromorphic families of invariant linear forms on analytic families of induced representations which can be used to analyze $\Hom_H(\pi,\BC)$ when $\pi$ is parabolically induced, and more generally a quotient of such a representation (see for example \cite{FLO}, \cite{MatZ}, \cite{MatJFA}, and \cite{SX}). They are also useful to establish results on multiplicities, see \cite{FLO} for a major application of this idea, and Theorem \ref{thm sl linear} below for a modest but interesting example. The natural question that we study in this paper, is under which condition on $\pi$ one can hope that $\Hom_H(\pi,\BC)$ can be exhausted by linear combinations of regularized or even normalized intertwining periods, for $\pi$ an induced representation. Whenever it is possible, it has nice applications as explained above, to the study of distinction of quotients of $\pi$ as well as that of multiplicities, but also to the computation of certain signs occurring in functional equations attached to the functionals in $\Hom_H(\pi,\BC)$ (see for example \cite{LM} or \cite{Sign}). 

The paper is organized as follows. In Section \ref{sec prelim}, we first recall generalities from \cite{LR}, \cite{OJNT}, \cite{GO} and \cite{CZlocal}, we define unimodular pairs, and verify that they are tempered, and that they induce unimodular pairs on stable Levi subgroups. In Section \ref{sec desc} we recall the basic results of \cite{MOY} on local intertwining periods, introduce the necessary terminology, and state our main conjecture, which is Conjecture \ref{conj main}. In Section \ref{sec indep} we prove some basic properties of local intertwining periods, which we use in the rest of the paper in specific situations, and which will be useful in later works as well. In Section \ref{sec group case}, we verify Conjecture \ref{conj main} in the group case, by showing that it boils down to a well-known result of Harish-Chandra on the commuting algebra of representations induced from square-integrable ones. In Section \ref{sec unitary}, for $E/F$ a quadratic extension of $p$-adic fields and $\U_n(E/F)$ a unitary subgroup of $\GL_n(E)$, we deduce Conjecture \ref{conj main} for the pair $(\GL_n(E),\U_n(E/F))$ from the results of \cite{FLO} and \cite{BPunitary}; the deduction requires a patient analysis of the results of \cite{FLO}. In Section \ref{sec inner GL}, we prove one part of the conjecture for pairs $(G,H)$ where $G$ is an inner form of $\GL_n$. In Section \ref{sec geom lemma and supp}, for $p$-adic fields, we prove again general results on open intertwining periods and their singularities, which are well-known to experts, and follow from the geometric lemma. Finally in Section \ref{sec small}, we prove some general results on local intertwining periods when $G$ has semi-simple rank one over a $p$-adic field, and apply them to prove Conjecture \ref{conj main} for pairs where $G$ is a special linear group of rank one. 

\tableofcontents

\section{Unimodular and tempered symmetric pairs}\label{sec prelim}

Let $F$ be a local field of characteristic zero field with normalized absolute value $|\ |_F$. Moreover when $F$ is Archimedean we assume that $F=\BR$, which is not a serious restriction: we see complex reductive groups as real reductive groups by a restriction of scalar argument. Let $G=\mathbf{G}(F)$ be an $F$-reductive group, and let $\theta$ be an involution of $G$ defined over $F$. We denote by $A_G$ its split component, i.e. the $F$-points of the connected center of $\mathbf{G}$. For any $A\subseteq G$ we set \[A^\theta=\{a\in A,\ \theta(a)=a\}\] and 
\[A^{\theta,-}=\{a\in A,\ \theta(a)=a^{-1}\}.\]
We set \[H:=G^{\theta},\] and call the pair  $(G,H)$ reductive \textit{symmetric pair}, or simply a symmetric pair. The set \[X=X_G:=G^{\theta,-}\] is called the \textit{symmetric space} attached to $(G,H)$, and it is equipped with the natural action of $G$ by twisted conjugation:
\[g\cdot x:=gx\theta(g)^{-1}.\] The map \[g\in G \to x_g:=g\cdot e\in X\] induces a homeomorphism between $G/H$ and the orbit of $e$ in $X$. For $g\in G$ and $x\in X$, we set $\Ad(g)(x)=gxg^{-1}$. For any $x\in X$, we can twist the involution $\theta$ by $x$:
\[\theta_x :=x\theta x^{-1}.\] We will mostly be interested by the twists by elements $x_g$ in $G\cdot e$, in which case 
\[\theta_{g}:=\theta_{x_g}=\Ad(g)\circ \theta  \circ \Ad(g)^{-1}.\] We observe that for $g$ in $G$, $A\subseteq G$, and $x\in X$: 
\[A^{\theta_g}=A\cap gHg^{-1},\] and \[\theta_{g\cdot x}=(\theta_x)_g.\] 
In particular \[G^{\theta_g}=gHg^{-1},\] and if $g'\in G$ then \[\theta_{gg'}=(\theta_{g'})_g.\]  We will freely and tacitly use these observations. 
Let $M$ be a Levi subgroup of $x\in X$ is $M$-\textit{admissible} if $\theta_x(M)=M$. We recall the very useful \cite[Lemma 6.3]{OJNT}. We observe that Offen works with parabolic subgroups standard with respect to a fixed  $\theta$-stable maximal split torus $T_0$ of $G$, which exists thanks to \cite[Lemma 2.4]{HW}, and $P_0$ a minimal (and not necessarily $\theta$-stable) parabolic subgroup containing it. Here we remove "standard" assumption.

\begin{lem}\label{lem stable decomp}
Suppose that $x\in X$ is $M$-admissible, and let $P$ be a parabolic subgroup of $G$ with $M$ as a Levi component, and $V$ as unipotent radical. Then 
\[P^{\theta_x}=M^{\theta_x}V^{\theta_x}.\]
\end{lem}
\begin{proof}
By \cite[Lemma 2.4]{HW}, and since $M$ is stable under $\theta_x$, 
there exists a $\theta_x$-stable maximally split torus $T_x$ inside $M$. By usual properties of the spherical building, we can fix $P_x$ a minimal parabolic subgroup of $G$ contained in $P$ and containing $T_x$. It now suffices to apply 
\cite[Lemma 6.3]{OJNT} with respect to  $\theta_x$, $T_x$, $P_x$ and the $M$-admissible point $e$ of $X$, in place of $\theta$, $T_0$, $P_0$ and $x$. 
\end{proof}

A simple situation where the above lemma applies is when $\theta_x(P)$ is opposite to  $P$ with respect to $M$, i.e. $\theta_x(P)\cap P=M$. In this case we say that $P$ is \textit{$\theta_x$-split} with respect to $M$, and we observe that  $V^{\theta_x}=\{e\}$.

\begin{lem}\label{lm adm open vs split}
Let $P=MV$ be a parabolic subgroup of $G$, and let $u\in G$. The following are equivalent:
\begin{enumerate}
	\item $P$ is $\theta_u$-split with respect to $M$.
	\item $PuH$ is open and $x_u$ is $M$-admissible.
\end{enumerate} 
\end{lem}
\begin{proof}
One direction is follows from \cite[Proposition 13.4]{HW}. The other as well but less directly. Assume that $x_u$ is $M$-admissible and that $PuH$ is open. We may assume that $u=e$. Let $\mathfrak{g}$ be the Lie algebra of $G$, $\mathfrak{p}$ that of $P$, $\mathfrak{m}$ that of $M$, $\mathfrak{v}$ that of $V$, so $\mathfrak{p}=\mathfrak{m}+\mathfrak{v}$. By abuse of notation we denote by $\theta$ again the differential of $\theta$ at $e$. Since  $\mathfrak{m}$ is $\theta$-stable, we also have $\theta(\mathfrak{p})=\mathfrak{m}+\theta(\mathfrak{v})$. Now since $PH$ is open, it follows from \cite[Proposition 13.4]{HW} that  $\mathfrak{g}=\mathfrak{p}+\theta(\mathfrak{p})=\mathfrak{v}+\mathfrak{m}+\theta(\mathfrak{v})$. But then for dimension reasons the sum must be direct, hence $\mathfrak{p}\cap \theta(\mathfrak{p})=\mathfrak{m}$.
\end{proof}

In \cite[Section 6]{OJNT} and following \cite{LR}, and after fixing a $\theta$-stable maximal split torus $T_0$ of $G$, and $P_0$ a minimal parabolic subgroup containing it as we just explained, Offen introduced an oriented graph, the vertices of which are the couples $(M,x)$ such that $M$ is a standard Levi subgroup of $G$, and $x\in X$ is $M$-admissible. We denote this graph by $\Gamma_G(\theta,T_0,P_0)$. Now we observe that the definition of $\Gamma_G(\theta,T_0,P_0)$ does not require that $T_0$ is $\theta$-stable, and that \cite[Lemma 6.4]{OJNT} still holds for 
$\Gamma_G(\theta,T_0,P_0)$ for a random choice of $T_0$. All that matters is that if $(M,x)$ is a vertex of $\Gamma_G(\theta,T_0,P_0)$, then $\theta_x(M)=M$ and hence $\theta_x$ acts on the roots of $A_M$ in $G$. We then define the graph $\Gamma_G(\theta)$ to be the oriented graph given by the disjoint union of the oriented graphs $\Gamma_G(\theta,T_0,P_0)$, for $T_0$ a maximal split torus of $G$, and $P_0$ a minimal parabolic subgroup of $G$ containing it. In particular a vertex could be labeled by a given pair $(M,x)$ such that $x$ is $M$-admissible more than one time, but it does not matter. 

\begin{df}
We call a pair $(M,x)$ such that $x$ is $M$-admissible a vertex (of $\Gamma_G(\theta)$). If $M$ is fixed, we sometimes simply say that $x$ is a vertex. 
\end{df}

Now we define the notion of \textit{unimodularity} for an $M$-admissible elements $x\in X$, which is referred to as the \textit{modulus assumption} in \cite{MOY}. 

\begin{df}
We say that the vertex $(M,x)$ is unimodular with respect to a parabolic subgroup $P$ of $G$ with Levi component $M$, if 
\begin{equation}\label{eqMA} \delta_{P^{\theta_x}}=(\delta_{P}^{1/2})_{|P^{\theta_x}}.\end{equation}
\end{df}

\begin{fct}\label{fact vertex 1}
Let $(M,x)$ be a vertex. Then to determine whether $x$ is unimodular with respect to $P$, it is enough to check that Equation \eqref{eqMA} holds on $M^{\theta_x}$. Moreover, if $m\in M$, then $(M,x)$ is unimodular with respect to $P$ if and only if $(M,m\cdot x)$ is unimodular with respect to $P$.  
\end{fct}
\begin{proof}
The first assertion is a consequence of Lemma \ref{lem stable decomp}. The second assertion easily follows from the fact that 
$P^{\theta_{m\cdot x}}=mP^{\theta_x}m^{-1}$. 
\end{proof}

\begin{df}
\begin{enumerate}
\item Suppose that $T_0\subseteq P_0$ is fixed. We say that $\Gamma_G(\theta,T_0,P_0)$ is unimodular if whenever $(M,x)$ is a vertex of $\Gamma_G(\theta,T_0,P_0)$, it is unimodular with respect to the unique parabolic subgroup $P$ of $G$ standard with respect to $P_0$ 
with Levi component $M$. 
\item We say that $\Gamma_G(\theta)$ is unimodular if for any vertex $(M,x)$ of $\Gamma_G(\theta)$ and any parabolic subgroup $P$ having $M$ as a Levi component, then $(M,x)$ is unimodular with respect to $P$. 
\end{enumerate}
\end{df}

\begin{fct}\label{fact vertex 2}
Suppose that $\Gamma_G(\theta,T_0,P_0)$ is unimodular for fixed $T_0\subseteq P_0$. Then $\Gamma_G(\theta)$ is unimodular. 
\end{fct}
\begin{proof}
Let $(M,x)$ be a vertex of $\Gamma_G(\theta)$. Then there exists $g\in G$ such that $gMg^{-1}$ is standard with respect to $(T_0,P_0)$. From the relation $\theta_{g\cdot x}=(\theta_x)_g$, we deduce that $(gMg^{-1},g\cdot x)$ is a vertex of $\Gamma_G(\theta,T_0,P_0)$. Take $P$ a parabolic subgroup of $G$ containing $M$. Then by assumption we have $\delta_{{(gPg^{-1})}^{\theta_{g\cdot x}}}=(\delta_{gPg^{-1}}^{1/2})_{|({gPg^{-1})}^{\theta_{g\cdot x}}}$, but as ${(gPg^{-1})}^{\theta_{g\cdot x}}=gP^{\theta_x}g^{-1}$, the conlcusion follows
from the relation 
$\delta_{gKg^{-1}}=\delta_K\circ \Ad(g^{-1})$ for any subgroup $K$ of $G$, and any $g\in G$.
\end{proof}

\begin{df}
We say that $(G,H)$ is a unimodular symmetric pair if $\Gamma_G(\theta)$ is unimodular. 
\end{df}

The unimodularity assumption is satisfied in many interesting cases. We refer to \cite{BM} for the terminology of pairs of Galois type, of PTB type, and of diagonal type (the latter corresponding to the so called group case).

\begin{lem}\label{lm ex mod} 
\begin{enumerate}
\item \label{A} If $(G,H)$ is either of diagonal type, of Galois type, of PTB type, or if $H/A_G \cap H$ is compact, then $(G,H)$ is unimodular. 
\item \label{B} Let $(M,x)$ be a vertex. Then it is unimodular with respect to any $\theta_x$-split parabolic subgroup $P$ of $G$ with Levi component $M$, i.e. $\theta_x(P)$ is opposite to $P$ and $M=P\cap \theta_x(P)$.
\item \label{C} More generally if $(M,x)$ is on the same connected component of $\Gamma(\theta)$ as a vertex above, it is unimodular with respect to the standard parabolic subgroup $P$ containing $M$ determined by this connected component.
\end{enumerate}
\end{lem}
\begin{proof}
The assertion for pairs of diagonal type is easy. For Galois pairs it follows from \cite[Proposition 4.3.2]{LR} or \cite[Corollary 6.9]{OJNT}, and Fact \ref{fact vertex 2}. It follows from \cite{Ch} and \cite[5.3]{BM} and  Fact \ref{fact vertex 2} for pairs of PTB type. For pairs with $H/A_G\cap H$ compact it is obvious. We prove the second point. By assumption $P^{\theta_x}=M^{\theta_x}$ and $M^{\theta_x}$ is reductive hence unimodular. From this we deduce that 
$\delta_{P^{\theta_x}}$ is trivial. We thus want to verify that $\delta_P$ is trivial on $M^{\theta_x}$, or equivalently on $A_M^{\theta_x}$. Now obviously the restriction of $\delta_P$ to $A_M^{\theta_x}$ is ${\theta_x}$-invariant. But on the other hand, if $U_P$ the unipotent radical of $P$, we have the formula \[\delta_P(a)=|\det(\Ad(a))_{|\mathrm{Lie}(U_P)}|_F\] for any $a\in A_M$. Now because $A_M$ is ${\theta_x}$-stable, the involution ${\theta_x}$ induces a bijection between the roots of $A_M$ in $\mathrm{Lie}(U_P)$ and the roots of $A_M$ in $\mathrm{Lie}(U_{{\theta_x}(P)})$. But because $P$ is ${\theta_x}$-split, the roots of $A_M$ in $\mathrm{Lie}(U_{{\theta_x}(P)})$ are the opposite of that of $A_M$ in $\mathrm{Lie}(U_P)$. This imples that \[\delta_P({\theta_x}(a))=\delta_P(a)^{-1}\] for any $a\in A_M$, and so $\delta_P$ is trivial on $A_M^{\theta_x}$. The last assertion follows from \cite[Lamma 6.4]{OJNT}, which as we observed, still holds for the graph $\Gamma_G(\theta)$.
\end{proof}

It is worthwhile observing that unimodular pairs have the nice property of being automatically \textit{tempered}, at least when $F$ is $p$-adic, though this should be true as well when $F$ is Archimedean. We recall that whenever $(G,H)$ is a symmetric pair, there exists a unique right $G$-invariant measure on $H\backslash G$, up to a positive scalar. With respect to such a measure, the space 
of functions $L^2(H\backslash G)$ becomes a unitary representation of $G$, and it admits a unique class of Plancherel measures. 

\begin{df}
We say that $(G,H)$ is a tempered symmetric pair if any Plancherel measure of $L^2(H\backslash G)$ is supported on the set of (isomorphism classes of) tempered representations of $G$.
\end{df}

Here we are following the terminology of \cite[Section 2.7]{BPGalsq}. Comparing \cite[Lemma 2.7.1]{BPGalsq} and \cite[Section 3.2]{CZlocal}, we see that tempered pairs are exactly the pairs called very strongly discrete in \cite{CZlocal}. In particular, by \cite[Section 3.2.2, (19)]{CZlocal}, we deduce the following result.

\begin{prop}\label{prop unimodular implies tempered}
Suppose that $F$ is $p$-adic and that $(G,H)$ is a unimodular symmetric pair. Then it is tempered.  
\end{prop}

We recall that by definition, an irreducible representation of $G$ is called \textit{square-integrable} if it is unitary, and if its matrix coefficients are square-integrable mod the center of $G$. We then remind that tempered pairs are strongly discrete in the terminology of \cite{GO}, as follows from \cite[Theorem 1.1]{GO} and \cite[Proposition 3.3]{CZlocal}, i.e.:

\begin{prop}\label{prop tempered implies strongly discrete}
Suppose that $F$ is $p$-adic and that $(G,H)$ is a tempered symmetric pair. Then any matrix coefficient of a square-integrable representation of $G$ belongs to $L^1(H/A_G\cap H)$. 
\end{prop}

Let $(G,H)$ be a tempered symmetric pair, and $\pi$ be a square-integrable representation of $G$. To each $v\in \pi$, and $v^\vee$ in the contragredient $\pi^\vee$ of $\pi$, we can associate the matrix coefficient 
\[c_{v,v^\vee}:g\to \langle \pi(g)v, v^\vee\rangle .\] Then by Proposition \ref{prop tempered implies strongly discrete}, the linear form 
\[\ell_{v^\vee}:v\to \int_{H/A_G\cap H} c_{v,v^\vee}(h)dh\] is a well-defined element in 
$\Hom_H(\pi,\BC).$ We set 
\[\mathcal{H}(\pi)=\{\ell_{v^\vee}, \ v^\vee \in \pi^\vee\}\subseteq \Hom_H(\pi,\BC).\]

We now state \cite[Theorem 1.4]{CZlocal}, which claims that all linear forms in $\Hom_H(\pi,\BC)$ can be canonically expressed in terms of matrix coefficients. 

\begin{thm}\label{thm H-dual discrete temp}
Suppose that $F$ is $p$-adic and that $(G,H)$ is a tempered symmetric pair. Then $\mathcal{H}(\pi)=\Hom_H(\pi,\BC)$. 
\end{thm}

We will mostly be interested in representations induced from square-integrable ones. Hence the following result is useful.

\begin{prop}\label{prop unimodular and Levi}
Let $(G,H)$ be a symmetric pair, and let $(M,x_g)$ be a vertex of $\Gamma_G(\theta)$ with $g\in G$. If $(G,H)$ is unimodular, then $(M,M^{\theta_g})$ is as well. 
\end{prop}
\begin{proof}
We observe that $(G,H)$ is unimodular if and only if $(G,gHg^{-1})$ is unimodular thanks to 
the relation $\theta_{g'g}=(\theta_g)_{g'}$ for $g$ and $g'\in G$ (see the proof of Fact \ref{fact vertex 2}). Hence we may assume that $g=e$ is the neutral element of $G$. 
Let $P=MU$ be a parabolic subgroup of $G$ with $M$ as Levi component. Any Levi subgroup $L$ of $M$ is of the form $Q\cap M$, for $Q=LV$ a parabolic subgroup of $G$ containing $P$. Letting $m\in M$, we want to prove that if $\theta_m(L)=L$, then $\delta_{{(Q\cap M)}^{\theta_m}}$ and $\delta_{Q\cap M}^{1/2}$ agree on $L^{\theta_m}$. However we have the following equality $\delta_Q=\delta_P\delta_{Q\cap M}$ on $L$. And 
also $\delta_Q^{1/2}=\delta_{Q^{\theta_m}}$ on $L^{\theta_m}$ and $\delta_P^{1/2}=\delta_{P^{\theta_m}}$ on $M^{\theta_m}$ since $(G,H)$ is unimodular. The first assertion now follows from the equality $\delta_{Q^{\theta_m}}=\delta_{P^{\theta_m}}\delta_{(Q\cap M)^{\theta_m}}$ on 
$L^{\theta_m}$, which is itself an easy consequence of the decomposition  
$Q^{\theta_m}=L^{\theta_m} V^{\theta_m}=L^{\theta_m}(V\cap M)^{\theta_m}U^{\theta_m}$ from Lemma \ref{lem stable decomp} (or rather its Lie algebra analogue).
\end{proof}

\section{Description of the problem}\label{sec desc}

 Let $F$ and $G$ be as in Section \ref{sec prelim}. Let $P$ be a parabolic subgroup of $G$, and $M$ a Levi component of $P$. 

Now let $\sigma$ be a (complex, smooth admissible) finite length representation of $M$. We denote by $I_P^G(\sigma)$ the representation of $G$ obtained from normalized parabolic induction of $\sigma$. We denote by $X^*(M)$ the lattice of algebraic characters of $M$, which is a lattice of rank $d$ which is the dimension of $A_M$. We set \[\frak{a}_{M}^*:=\BR \otimes_{\BZ} X^*(M)\simeq \BR^d\] and \[\frak{a}_{M,\BC}^*:=\BC \otimes_{\BZ} X^*(M)=\frak{a}_{M}^*+i\frak{a}_{M}^*\simeq \BC^d.\] Hence for $\us\in \frak{a}_{M,\BC}^*$, its real part $\re(\us)\in \frak{a}_{M}^*$ is well defined. Then we have a natural homomorphism 
\[\Phi:\frak{a}_{M,\BC}^*\to \Hom(M,\BC^\times)\] acting on pure tensors by the formula
\[\Phi(s\otimes \chi)=|\chi(m)|_F^s.\] We denote by \[X^0(M):=\ima(\Phi)\] the image of this map, which we call the group of unramified characters of $M$. In particular elements in $X^0(M)$ are naturaly parametrized by vectors $\underline{s}\in \frak{a}_{M,\BC}^*$, and more precisely we set 
\[\chi_{\underline{s}}:=\Phi(\underline{s}).\] For $\underline{s}\in \frak{a}_M^*$ we set 
\[\sigma[\us]:=\chi_{\us}\otimes \sigma.\]
 We fix for the rest of this section a maximal compact subgroup $K$ of $G$ which is in \textit{good position} with respect to $(P,M)$ in the following sense:

\begin{enumerate}
\item $G=PK$.
\item $M\cap K$ is a maximal compact subgroup of $M$.
\item If $P=MU$ is the Levi decomposition of $P$ with respect to $M$, then $P\cap K=(M\cap K)(U\cap K).$ Note that this actually automatically follows from the second condition. 
\end{enumerate}

 Then to $f\in  I_P^G(\sigma)$, one can attach a holomorphic section 
\[f_{\us}\in I_P^G(\sigma[\us])\] where $\us$ varies in $\frak{a}_M^*$. Explicitly putting 
\[\eta_{\us}(umk)=\chi_{\us}(m)\] for $u\in U$, $m \in M$ and $k\in K$, then $f_{\us}=\eta_{\us}f$. 

Now suppose that $(G,H)$ is a symmetric pair associated to the involution $\theta$ of $G$. Write 
\[G=\coprod_{u\in R(P\backslash G /H)} PuH.\] Throughout the paper we will always make the following assumption on $R(P\backslash G/ H)$. 

\begin{ass}
The representatives $u\in R(P\backslash G /H)$ are chosen such that whenever $PuH$ contains $puh$ such that $(M,x_{puh})$ is a vertex, then 
$(M,x_u)$ is a vertex itself. 
\end{ass}

 For any vertex $(M,x_u)$ (where $u\in R(P\backslash G /H)$ always), we denote by $X^0(M)^{\theta_u,-}$ the subgroup of $\theta_u$-anti-invariant characters inside $X^0(M)$:

\[X^0(M)^{\theta_u,-}:=\{\chi\in X^0(M),\ \chi\circ \theta_u=\chi^{-1}\}.\]

The group $X^0(M)^{\theta_u,-}$ is the image under the map $\Phi$ of the complex vector space 
\[\frak{a}_{M,\BC}^*(\theta_u,-1):=\{\underline{s}\in \frak{a}_{M,\BC}^*,\ \theta_u(\underline{s})=-\underline{s}\},\] and we set 
\[\frak{a}_{M}^*(\theta_u,-1):=\frak{a}_{M,\BC}^*(\theta_u,-1)\cap \frak{a}_{M}^*.\]

Now suppose that the vertex $(M,x_u)$ is unimodular with respect to $P$, and fix an invariant linear form in 
\[\ell\in \Hom_{M^{\theta_u}}(\sigma,\BC).\]

It is proved in \cite{MOY} (which considers in fact any vertex $(M,x)$ unimodular with respect to $P$) that if $f_{\us}$ is a holomorphic section of $I_P^G(\sigma[\us])$, the integral 
\[\int_{u^{-1}Pu\cap H\backslash H} \ell(f_{\us}(uh))dh\] is formally well-defined because $x_u$ is unimodular, and convergent for $\re(\us)$ in $\mathcal{D}_{x_u,\sigma}$ where $\mathcal{D}_{x_u,\sigma}\subseteq \frak{a}_M^*(\theta_u,-1)$ is a non empty open cone independent of $f$. This defines a linear form 
\[J_{x_u,\sigma,\ell,\us}:f_{\us}\to \int_{u^{-1}Pu\cap H\backslash H} \ell(f_{\us}(uh))dh\] in $\Hom_H(I_P^G(\sigma[\us]),\BC)$ whenever 
$\re(\us)\in \mathcal{D}_{x_u,\sigma}$. It is then proved in \cite{MOY} that this family of linear forms extends meromorphically to $\us\in \frak{a}_{M,\BC}^*(\theta_u,-1)$ in the following sense: for each $\us_0\in \frak{a}_{M,\BC}^*(\theta_u,-1)$, there exists a nonzero meromorphic function $F_{\us_0}(\us)$ on $\frak{a}_{M,\BC}^*(\theta_u,-1)$ such that $F_{\us_0}(\us)J_{x_u,\sigma,\ell,\us}(f_{\us})$ is holomorphic at $\us_0$ for each $f\in  I_P^G(\sigma)$. When $F$ is $p$-adic we can actually choose $F_{\us_0}$ independent of $\us$, and polynomial in the variable 
$(q^{\pm s_1},\dots, q^{\pm s_d})$ for $q$ the residual characteristic of $F$, where the $s_i$ pertain to the choice of a basis of $\frak{a}_{M,\BC}^*$. Finally if $\ell\neq 0$, then the intertwining period $J_{x_u,\sigma,\ell,\us}(f_{\us})$ is nonzero for at least one $f\in  I_P^G(\sigma)$.

\begin{rem}\label{rem indep}
The statements in \cite{MOY} assume that the parabolic subgroup $P$ of $G$ is standard with respect to a fixed minimal parabolic subgroup $P_0$ containing a $\theta$-stable maximal split torus $T_0$, and that $M$ contains $T_0$. Moreover in \cite[Section 2.2]{MOY}, the maximal compact subgroup $K$ is assumed to be in ``very good position" with respect to $P_0$ and $T_0$. For example when $F=\BR$, though not expicitly stated in \cite{MOY}, one checks that $\theta_c(P)$ is opposite to $P$ for any parabolic subgroup $P$ containing $P_0$, for the Cartan involution $\theta_c$ commuting with $\theta$ selected in \cite[Section 2.1]{MOY}, and having $K$ as fixed points. Indeed $T_0$ being furthermore $\theta_c$-stable in \cite[Section 2.2]{MOY}, and because it is maximally split over $\BR$, the involution $\theta_c$ necessarily acts as $t\to t^{-1}$ on $T_0$, and hence sends the relative root subgroup $U_{\alpha}$ to $U_{-\alpha}$, for any relative root of $T_0$ in the Lie algebra of $G$. In particular the maximal compact subgroup $K$ fixed by $\theta_c$ is in good position with repsect to $(P,M)$ in \cite{MOY}. In Section \ref{sec indep}, we explain how the general setting of Section \ref{sec desc} reduces to the more specific setting of \cite{MOY}. 
\end{rem} 

From now on, we assume that $(G,H)$ is unimodular. If $\ell\neq 0 \in \Hom_{M^{\theta_u}}(\sigma,\BC)$, then for a generic choice of $\us_0\in \frak{a}_{M,\BC}^*(\theta_u,-1)-\{0\}$, which means for $\us_0$ outside of a countable union of hyperplanes in $\frak{a}_{M,\BC}^*(\theta_u,-1)$, there exists $k(\us_0)\in \BZ$ such that the linear map 
\begin{equation}\label{eq reg} J_{x_u,\sigma,\ell}^{*,\us_0}:=\lim_{s\to 0} s^{k(\us_0)}J_{x_u,\sigma,\ell,s\times \us_0}\end{equation} is nonzero. 

\begin{df}
\begin{itemize}
\item We call an element of the form \[J_{x_u,\sigma,\ell}^{*,\us_0}\in \Hom_H(I_P^G(\sigma),\BC)-\{0\}\] a \textit{regularized intertwining period}. 
\item We denote by 
\[\Hom_H^*(I_P^G(\sigma),\BC)\] the subspace of $\Hom_H(I_P^G(\sigma),\BC)$ spanned by regularized intertwining periods.
\end{itemize}
\end{df}

\begin{rem}\label{rem indep rep adm} Note that if $(M,x_u)$ is a vertex, then the elements $u'\in  PuH$ such that $(M,x_{u'})$ is a vertex is exactly the set 
	$MuH$, as follows from \cite[Corollary 6.2]{OJNT}. Hence  $\frak{a}_{M,\BC}^*(\theta_u,-1)=\frak{a}_{M,\BC}^*(\theta_{u'},-1)$ since $M$ acts trivially on $\frak{a}_{M,\BC}^*$. Moreover if $u'=muh$, then $M^{\theta_{u'}}=mM^{\theta_u}m^{-1}$ and $\Hom_{M^{\theta_{u'}}}(\sigma,\BC) \simeq \Hom_{M^{\theta_u}}(\sigma,\BC)$ via $\ell\to \ell\circ \pi(m)$. In particular the space of regularized intertwining periods does not depend on the choice of $u$ such that $x_u$ is $M$-admissible.\end{rem}

Now to each vertex $(M,x_u)$,  we attach a meromorphic function $\mathfrak{n}(x_u,\sigma,\us)$ on $\frak{a}_{M,\BC}^*(\theta_u,-1)$, which we call a normalizing factor. 

\begin{df}
\begin{itemize}
\item We call an element of the form \[\mathfrak{J}_{x_u,\sigma,\ell,\us}:=\mathfrak{n}(x_u,\sigma,\us)J_{x_u,\sigma,\ell, \us}\] a \textit{normalized intertwining period} attached to our fixed family of normalizing factors. 
\item We denote by 
\[\Hom_H^{\mathfrak{n}}(I_P^G(\sigma),\BC)\] the subspace of $\Hom_H(I_P^G(\sigma),\BC)$ spanned by the values at $\us=\underline{0}$ of the normalized intertwining periods which are holomorphic at $\us=\underline{0}$. 
\end{itemize}
\end{df}

We recall from \cite{VDBfinite} and \cite{Delfinite} that the space $\Hom_H(I_P^G(\sigma),\BC)$ is finite dimensional whenever $(G,H)$ is a symmetric pair, for $\sigma$ of finite length. The goal of this paper is to study the following question. 

\begin{q*}
Suppose that $(G,H)$ is a unimodular symmetric pair. Let $\sigma$ be an irreducible representation of $G$. When is the space $\Hom_H(I_P^G(\sigma),\BC)$ spanned by the regularized intertwining periods? When does the equality \begin{equation}\label{eq main} \Hom_H(I_P^G(\sigma),\BC)=\Hom_H^*(I_P^G(\sigma),\BC)\end{equation} hold?
\end{q*}

A positive answer to the above question can be useful. For example in \cite{MatJFA} and \cite{SX} it is used to study distinction of quotients of induced representations, such as discrete series or Speh representations, whereas in \cite{Sign} it is used to compute the sign of linear periods. We hope that the answer is yes in the following case. We recall that $\sigma$ is called \textit{square-integrable} if it is unitary, and its matrix coefficients are square-integrable mod the center of $M$. 

\begin{conjecture}\label{conj main}
\begin{enumerate}[(i)]
\item\label{conj a} Equality \eqref{eq main} holds whenever $\sigma$ is a square-integrable representation of $M$.
\item\label{conj b} Actually, when $\sigma$ is square-integrable, there exists a family of normalizing factors $\mathfrak{n}(x_u,\sigma,\us)$such that \[\Hom_H(I_P^G(\sigma),\BC)=\Hom_H^{\mathfrak{n}}(I_P^G(\sigma),\BC).\] 
\end{enumerate}
\end{conjecture}

\begin{rem}\label{rem norm int other prop}
The normalization factors in Conjecture \ref{conj main}\eqref{conj b} could probably be defined whenever $\sigma$ is irreducible, and required to satisfy other basic properties, especially with respect to some functional equations of local intertwining periods with respect to normalized intertwining operators. We leave this investigation to somewhere else.
\end{rem}

\begin{rem}\label{rem norm in other prop}
In view of Theorem \ref{thm H-dual discrete temp} and Propositions \ref{prop unimodular implies tempered}, \ref{prop tempered implies strongly discrete}, and \ref{prop unimodular and Levi}, Conjecture \ref{conj main} gives a canonical description of $\Hom_H(I_P^G(\sigma),\BC)$ at least when $F$ is $p$-adic, up to the fact that the normalization factors can have a canonical description.
\end{rem}

In the following remark, we discuss a possible extension of the conjecture to more general pairs.

\begin{rem}\label{rem non unimodular case}
Suppose that $(M,x_u)$ is any vertex, not necessarily unimodular with respect to $P$. We set 
\[\delta_{x_u}:= \delta_{P^{\theta_u}}\times \delta_{P}^{-1/2},\] and view it as a character of $M^{\theta_u}$. Then, for \[\ell\in \Hom_{M^{\theta_u}}(\sigma,\delta_{x_u}),\] it is formally possible to consider the intertwining period 
\[J_{x_u,\sigma,\ell,\us}(f_{\us}):=\int_{u^{-1}Pu\cap H\backslash H} \ell(f_{\us}(uh))dh\] for $f_{\us}$ a holomorphic section of 
$I_P^G(\sigma[\us])$. However \cite{MOY} fails to prove that such an integral converges for $\us$ in an open cone, and it would be interesting to check if the proof of convergence there can be extended without too much efforts to the setting of tempered symmetric pairs. On the other hand 
the proof of \cite[Theorem 5.4]{MOY} provides a meromorphic family of $H$-invariant linear forms on $I_P^G(\sigma[\us])$, which agrees with usual intertwining periods when $x_u$ is unimodular, and would agree with the meromorphic continuation of the above general integrals if they were known to converge in a cone when $x_u$ is not unimodular. So denoting by $J_{x_u,\sigma,\ell,\us}$ the ``intertwining periods" defined in the proof of \cite[Theorem 5.4]{MOY}, it is again possible to make Conjecture \ref{conj main} with respect to these general intertwining periods. It is plausible that it could hold for tempered symmetric pairs. For the non unimodular examples which we are familiar with, which are all tempered, Conjecture \ref{conj main}\eqref{conj a} indeed still holds. See Remark \ref{rem linear model}. 
\end{rem}

\section{Basic properties of intertwining periods}\label{sec indep}

In this section $(G,H)$ is unimodular. We check that the basic properties of intertwining periods (see Remark \ref{rem indep}), as well the statement of Conjecture \ref{conj main} are independent of some choices. For example, in Section \ref{sec desc}, we used a fixed maximal compact subgroup of $G$ in good position with respect to $(P,M)$ to define holomorphic sections. Here we observe that if Conjecture \ref{conj main} holds for one choice of maximal compact subgroup well positioned with respect to $(P,M)$, then it holds for any. Then we also prove that Conjecture \ref{conj main} does not depend on the conjugacy class of $H$ inside $G$, in the sense that if it true for $H$, it is automatically true for $gHg^{-1}$ whenever $g\in G$. Then we check a generalization of this property, which for example applies when $G$ is a special linear group but $g$ above belongs to the corresponding general linear group. The reason in each case is a similar argument. 

\subsection{Intertwining periods and $G$-conjugacy of $H$ and $(P,M)$}\label{sec indep conj1}

Here we check that to claim the properties of intertwining periods stated in Section \ref{sec desc}, we may assume that $P$ is standard with respect to a fixed minimal parabolic subgroup $P_0$ containing a maximal split torus $T_0$ which is $\theta$-stable, and that $M$ is standard with respect to $T_0$, as in \cite{MOY}. Before that, we check that Conjecture \ref{conj main} does not depend on the $G$-conjugacy class of $H$. 

First we verify that if Conjecture \ref{conj main} holds for the unimodular symmetric pair $(G,H)$, it automatically holds for $(G,gHg^{-1})$, and even more. Let $\sigma$ be a finite length representation of $M$. We set $H':=gHg^{-1}$, and observe that it is the group of fixed points of the involution $\theta_g$. Clearly the spaces $\Hom_{H}(I_P^G(\sigma),\BC)$ and $\Hom_{gHg^{-1}}(I_P^G(\sigma),\BC)$ have the same (finite) dimension, as the map $L\to L\circ \pi(g)^{-1}$ provides an isomorphism between them (where $\pi=I_P^G(\sigma)$). Now if 
\[G=\coprod_{u\in R(P\backslash G /H)} PuH,\] then 
\[G=\coprod_{u\in R(P\backslash G /H)} Pug^{-1}H'.\]

Moreover we recall the relation $\theta_u=(\theta_g)_{u g^{-1}},$ hence for any vertex $(M,x_u)$ we have 
\[\Hom_{M^{\theta_u}}(\sigma,\BC)=\Hom_{M^{(\theta_g)_{ug^{-1}}}}(\sigma,\BC)\] and \[\frak{a}_{M,\BC}^*(\theta_u,-1)=\frak{a}_M^*((\theta_g)_{ug^{-1}},-1).\] 
For $\re(\us)$ in an appropriate open cone of 
$\frak{a}_{M,\BC}^*(\theta_u,-1)$, we have the equality of convergent integrals 
\[\int_{u^{-1}Pu\cap H\backslash H} \ell(f_{\us}(uh))dh=\int_{gu^{-1}Pug^{-1}\cap H'\backslash H'} \ell(f_{\us}(ug^{-1}h'g))dh'\] for any holomorphic section $f_{\us} \in I_P^G(\sigma)$ with respect to $K$. We thus deduce the equality 
\begin{equation} \label{eq rel int conj} J_{x_u,\sigma,\ell,\us}^H(f_{\us})=J_{x_{ug^{-1}},\sigma,\ell,\us}^{H'}(\rho(g)f_{\us}),\end{equation}
and conversely 
\[J_{x_{ug^{-1}},\sigma,\ell,\us}^{H'}(f_{\us})=J_{x_u,\sigma,\ell,\us}^H(\rho(g^{-1})f_{\us}),\] where $\rho$ stands for the right translation. 

Observing that both $\rho(g)f_{\us}$ and $\rho(g^{-1})f_{\us}$ are a holomorphic combinations of holomorphic sections in $I_P^G(\sigma[\underline{s}])$, we deduce the following.

\begin{prop}\label{prop indep conj1}
Let $\sigma$ be a finite length representation of $M$. Then 
\[\Hom_H(I_P^G(\sigma),\BC)=\Hom_H^*(I_P^G(\sigma),\BC)\iff \Hom_{H'}(I_P^G(\sigma),\BC)=\Hom_{H'}^*(I_P^G(\sigma),\BC).\] Moreover for any family of normalization factors $\mathfrak{n}(x_u,\sigma,\us)$, setting $\mathfrak{n'}(x_{ug^{-1}},\sigma,\us):=\mathfrak{n}(x_u,\sigma,\us)$, one has 
\[\Hom_H(I_P^G(\sigma),\BC)=\Hom_H^{\mathfrak{n}}(I_P^G(\sigma),\BC)\iff \Hom_{H'}(I_P^G(\sigma),\BC)=\Hom_{H'}^{\mathfrak{n'}}(I_P^G(\sigma),\BC)\] for any family of normalizing factors. In particular Conjecture \ref{conj main} holds for the unimodular pair $(G,H)$ if and only if it holds for $(G,H')$. 
\end{prop}

Now if $P'=gPg^{-1}$ and $M'=gMg^{-1}$, we observe that if $K$ is our chosen maximal compact subgroup in good position with respect to $(P,M)$, then $K'=gKg^{-1}$ is in good position with respect to $(P',M')$. 
 Moreover if $f_{\us}^K$ is holomorphic section of $I_P^G(\sigma[\us])$ with respect to $K$, then 
\[f_{\us}^{K'}:=\lambda(g^{-1})\rho(g)f_{\us}^K=f_{\us}^K(g^{-1} \ \bullet \ g)\] is a holomorphic section of $I_{P'}^G(\sigma^g[\us])$ with respect to $K'$, where $\sigma^g:=\sigma(g^{-1} \ \bullet \ g)$ and $\lambda$ stands for the left translation. One has the decomposition 
\[G=\coprod_{u\in R(P\backslash G /H)} P'guH,\] and $(M,x_u)$ is a vertex if and only if $(M',x_{gu})$ is a vertex. Hence from the equality 
\[\int_{u^{-1}Pu\cap H\backslash H} \ell(f_{\us}^K(uh))dh=\int_{u^{-1}g^{-1}P'gu\cap H\backslash H} \ell(f_{\us}^{K'}(guhg^{-1}))dh,\] and using the fact that $\rho(g^{-1})f_{\us}^{K'}$ is a holomorphic combination of holomorphic sections with respect to $K'$, we deduce from the above discussion and the next section on the independence of well-positioned maximal compact subgroups, that we could choose the pair $(P,M)$ such that $M$ contains $T_0$ and $P$ contains $P_0$ to state the basic properties of intertwining periods in Section \ref{sec desc}, as it is always conjugate to such a pair. 

\subsection{Intertwining periods and maximal compact subgroups}\label{sec indep comp} 

Suppose that $K$ and $K'$ are two maximal compact subgroups of $G$ in good position with respect to $(P,M)$. In order to define $J_{x_u,\sigma,\ell,\us}$ and its regularization along generic directions, we fixed a compact subgroup $K$ as above. First we need to justify that the basic properties of $J_{x_u,\sigma,\ell,\us}$ stated in Section \ref{sec desc} are independent on this choice, and then we need to claim for their regularizations along a generic direction give the same $H$-invariant linear form up to a nonzero scalar. These facts follow from the fact that if $f_{\us}^K$ is a holomorphic section of $I_P^G(\sigma[\us])$ with respect to $K$, then it is a holomorphic combination of holomorphic sections with respect to $K'$, and conversely.

\subsection{Intertwining periods and  $\tilde{G}$-conjugacy}\label{sec indep conj2}

Let $\mathbf{G}$ be the algebraic reductive group defined over $F$ such that $G=\mathbf{G}(F)$. In this section we make the following further assumption:

\begin{ass}
There exists an $F$-reductive group $\tilde{\mathbf{G}}$ containing $\mathbf{G}$ such that $\mathbf{G}$ is the derived subgroup of $\tilde{\mathbf{G}}$. 
\end{ass}

We set $\tilde{G}=\tilde{\mathbf{G}}(F)$. Then the map $\tilde{P}=\tilde{\mathbf{P}}(F)\to P:=(\tilde{\mathbf{P}}\cap \mathbf{G})(F)$ is a bijection from the set of parabolic subgroups of $\tilde{G}$ to that of parabolic subgroups of $G$. Fixing $\tilde{P}$ a parabolic subgroup of $\tilde{G}$, then the map $\tilde{M}=\tilde{\mathbf{M}}(F)\to M:=(\tilde{\mathbf{M}}\cap \mathbf{G})(F)$ is a bijection from the set of Levi components of $\tilde{P}$ to that of Levi components of $P$. Furthermore:

\begin{enumerate}
\item $M$ is a normal subgroup of $\tilde{M}$.
\item $\tilde{M}=\tilde{M_{0}}M$ for one (hence for any) minimal Levi subgroup $\tilde{M_{0}}$ of $\tilde{M}$. 
\item $\frac{\tilde{M}}{A_{\tilde{M}}M}$ is a finite abelian group. 
\end{enumerate}

We write $F^{\tilde{g}}(x)=F(\tilde{g}^{-1} x \tilde{g})$ whenever $\tilde{g}\in \tilde{G}$ and $F$ is a map on a set $X$ contained in $G$. The first result that we want to prove here is the following: 

\begin{prop}\label{prop GL conj in SL}
Let $\sigma$ be a finite length representation of $M$, let $\tilde{g}=\tilde{m}g$, $\tilde{m}\in M$, $g\in G$, belong to $\tilde{G}$, and set 
$H':=\tilde{g} H \tilde{g}^{-1}$, so that \[\Hom_H(I_P^G(\sigma),\BC)\simeq \Hom_{H'}(I_P^G(\sigma)^{\tilde{g}},\BC)\simeq \Hom_{H'}(I_P^G(\sigma^{\tilde{m}}),\BC).\] Then 
\[\Hom_H(I_P^G(\sigma),\BC)=\Hom_H^*(I_P^G(\sigma),\BC) \iff \Hom_{H'}(I_P^G(\sigma^{\tilde{m}}),\BC)=\Hom_{H'}^*(I_P^G(\sigma^{\tilde{m}}),\BC).\]
Moreover for any family of normalization factors $\mathfrak{n}(x_u,\sigma,\us)$, setting \[\mathfrak{n'}(x_{\tilde{m}ug^{-1}\tilde{m}^{-1}},\sigma,\us):=\mathfrak{n}(x_u,\sigma,\us),\] one has 
\[\Hom_H(I_P^G(\sigma),\BC)=\Hom_H^\mathfrak{n}(I_P^G(\sigma),\BC) \iff \Hom_{H'}(I_P^G(\sigma^{\tilde{m}}),\BC)=\Hom_{H'}^\mathfrak{n'}(I_P^G(\sigma^{\tilde{m}}),\BC).\]

In particular Conjecture \ref{conj main} holds for the unimodular pair $(G,H)$ if and only if it holds for the pair $(G,\tilde{g}H\tilde{g}^{-1})$. 

Moreover if $\sigma$ extends to $\tilde{M}$ so that $\Hom_H(I_P^G(\sigma),\BC)\simeq \Hom_{H'}(I_P^G(\sigma) ,\BC),$ then \[\Hom_H(I_P^G(\sigma),\BC)=\Hom_H^*(I_P^G(\sigma),\BC) \iff \Hom_{H'}(I_P^G(\sigma) ,\BC)=\Hom_{H'}^*(I_P^G(\sigma) ,\BC),\] and the same holds with $\Hom_H^{\mathfrak{n}}(I_P^G(\sigma),\BC)$ in place of $\Hom_H^*(I_P^G(\sigma),\BC)$ and $\Hom_{H'}^{\mathfrak{n'}}(I_P^G(\sigma),\BC)$ in place of $\Hom_{H'}^*(I_P^G(\sigma),\BC)$. 
\end{prop}
\begin{proof}
We set $H':=\tilde{g}H\tilde{g}^{-1}$. According to Proposition \ref{prop indep conj1}, we may assume that $K=\tilde{K}\cap G$ such that $\tilde{K}$ is as in Lemma \ref{lm double good position}. Moreover because $\tilde{g}=\tilde{m}g$ for $g\in G$ and $\tilde{m}$ in $\tilde{M}$, we may assume thanks to Section \ref{sec indep conj1} that $\tilde{g}=\tilde{m}$. In particular $H'=G^{\theta_{\tilde{m}}}$. 
 First we observe that 
\[I_P^G(\sigma)^{\tilde{m}}\simeq I_P^G(\sigma^{\tilde{m}}).\] From this we deduce that 
\[\Hom_H(I_P^G(\sigma),\BC)\simeq \Hom_{H'}(I_P^G(\sigma^{\tilde{m}},\BC)).\]
Now if 
\[G=\coprod_{u\in R(P\backslash G /H)} PuH,\] because $\tilde{m}$ normalizes $P$, we have 
\[G=\coprod_{u\in R(P\backslash G /H)} P\tilde{m}u\tilde{m}^{-1}H'.\] We fix $u$ such that $(M,x_u)$ is a vertex. Then observe that 
\[(\theta_{\tilde{m}})_{\tilde{m}u\tilde{m}^{-1}}=\theta_{\tilde{m}u}=(\theta_{u})_{\tilde{m}}.\] In particular 
\[\frak{a}_{M,\BC}^*((\theta_{\tilde{m}})_{\tilde{m}u\tilde{m}^{-1}},-1)=\frak{a}_{M,\BC}^*((\theta_{u})_{\tilde{m}},-1))=\frak{a}_{M,\BC}^*(\theta_u,-1),\] since 
$\tilde{M}$ atcs trivially on $\frak{a}_{M,\BC}^*\subseteq \frak{a}_{\tilde{M},\BC}^*$. Also 
\[\tilde{m} M^{\theta_u}\tilde{m}^{-1}= M^{\theta_{\tilde{m}u}}.\] Obviously the identity map 
$\ell\to \ell$ is an isomorphism between $\Hom_{M^{\theta_u}}(\sigma,\BC)$ and 
$\Hom_{M^{\theta_{\tilde{m}u}}}(\sigma^{\tilde{m}},\BC)$. Take $f_{\us}$ a holomorphic section of $I_P^G(\sigma[\us])$ with respect to $K$ for $f_{\us}$. 

Then 
for $\us$ in an appropriate open cone of 
$\frak{a}_{M,\BC}^*(\theta_u,-1)$, we have the equality of convergent integrals 
\begin{equation*}
	\begin{split}
	&	\int_{u^{-1}Pu\cap H\backslash H} \ell(f_{\us}(uh))dh\\
		&= \int_{\tilde{m}u^{-1}\tilde{m}^{-1}P\tilde{m}u\tilde{m}^{-1}\cap H'\backslash H'} 
		\ell(f_{\us}^{\tilde{m}}((\tilde{m}u\tilde{m}^{-1})h')dh
	\end{split}
\end{equation*} where $f_{\us}^{\tilde{m}}$ is a holomorphic section of $I_P^G(\sigma^{\tilde{m}}[\us])$ with respect to $\tilde{m}K\tilde{m}^{-1}$. Hence
\[J_{x_u,\sigma,\ell,\us}^H(f_{\us})=J_{x_{\tilde{m}u\tilde{m}^{-1}},\sigma,\ell,\us}^{H'}(f_{\us}^{\tilde{m}}).\]
We then conclude the first statement of the proposition thanks to the arguments in Section \ref{sec indep comp}. The last statement is now clear. The third one follows from the fact that if $\sigma$ extends to a representation $\tilde{\sigma}$ of $\tilde{M}$, then $I_P^G(\sigma)$ is the restriction to $G$ of $I_{\tilde{P}}^{\tilde{G}}(\tilde{\sigma})$, and $I_P^G(\sigma)^{\tilde{g}}$ is the restriction to $G$ of $I_{\tilde{P}}^{\tilde{G}}(\tilde{\sigma})^{\tilde{g}}$, but $I_{\tilde{P}}^{\tilde{G}}(\tilde{\sigma})^{\tilde{g}}\simeq I_{\tilde{P}}^{\tilde{G}}(\tilde{\sigma})$ hence 
$I_P^G(\sigma)^{\tilde{g}}\simeq I_P^G(\sigma)$. 
\end{proof}

We now want to prove a related but slightly different result. We will make use of the following fact.

\begin{lem}\label{lm double good position}
Fix $P$ (resp. $M$) a parabolic subgroup of $G$ (resp. a Levi component of $P$). Let $\tilde{P}$ be the parabolic subgroup of $\tilde{G}$ such that $P=G\cap \tilde{P}$ and $\tilde{M}$ its Levi subgroup such that $M=G\cap \tilde{M}$. Then there exists $\tilde{K}$ a maximal compact subgroup of $\tilde{G}$ in good position with respect to $(\tilde{P},\tilde{M})$ such that $K:=\tilde{K}\cap G$ is a maximal compact subgroup of $G$ in good position with respect to $(P,M)$.
\end{lem}
\begin{proof}
When $F=\BR$, we choose $\tilde{T_0}$ a maximal split torus of $\tilde{G}$ stable under a Cartan involution $\theta_c$. Then $\theta_c$ restricts to $G$ as a Cartan involution fixing the maximal split torus $T_0$ of $G$ contained in $\tilde{T_0}$. Recall that we have observed in Remark \ref{rem indep} that $\theta_c$ acts as the inversion on $\tilde{T_0}$ hence exchanges $\widetilde{U_\alpha}=U_\alpha$ with $\widetilde{U_{-\alpha}}=U_{-\alpha}$ for any root of $T_0$ in the Lie algebra of $G$. We choose $P_0$ a minimal parabolic subgroup containing $T_0$. Then the maximal compact subgroup $\tilde{K}$ of $\tilde{G}$ fixed by $\theta_c$ is in good position with respect to any $(\tilde{P},\tilde{M})$ such that both $\tilde{P}$ and $\tilde{M}$ are semi-standard, and $K$ is also in good position with respect to $(P,M)$. Furthermore we may assume, after $\widetilde{G}$-conjugacy, that our original pairs $(\tilde{P},\tilde{M})$ and $(P,M)$ are semi-standard. 
When $F$ is $p$-adic, we fix $\tilde{T_0}$ and $T_0$ as in the real case. The Bruhat-Tits buildings $B_{\tilde{G}}$ of $\widetilde{G}$ and $B_G$ of $G$ canonically identify (see \cite[Section 4.1]{KP}). If we take $(\tilde{P},\tilde{M})$ semi-standard with respect to $\tilde{T_0}$, so that $(P,M)$ is semi-standard with respect to $\tilde{T_0}$, a special vertex $x$ in the apartment of $B_G$ corresponding to $T_0$, and $\tilde{x}$ the special vertex in the apartment of $B_G$ corresponding to $T_0$, which identifies with $x$, then the stabilizer $\tilde{K}$ of $\tilde{x}$ satisifies the expected properties with respect to $(\tilde{P},\tilde{M})$ and $(P,M)$: this follows from \cite[Section 3.11]{Pudesc} and \cite[Theorem 5.3.4]{KP}. Again, after $\widetilde{G}$-conjugacy, this gives the desired result. 
\end{proof}

We recall that there is a canonical identification $\frak{a}_{\tilde{M}}^*\simeq \frak{a}_M^*\oplus \frak{a}_{A_{\tilde{M}}}^*$. Now fix $\sigma$ a finite length representation of $M$, and assume that there exists $\tilde{\sigma}$ a finite length representation of $\tilde{M}$ such that $\sigma=\tilde{\sigma}_{|M}$. For $\us\in\frak{a}_M^*$, the restriction of functions to $G$ induces a $G$-module surjection 
from $I_{\tilde{P}}^{\tilde{G}}(\tilde{\sigma} [\us])$ to $I_P^G(\sigma[\us])$. Moreover we have the following obvious fact. 

\begin{fct}\label{fact hol sec}
Fix $\tilde{K}$ as in Lemma \ref{lm double good position}, and let $\sigma$ and $\tilde{\sigma}$ be as above. Then the restriction to $G$ map from $I_{\tilde{P}}^{\tilde{G}}(\tilde{\sigma} [\us])$ to $I_P^G(\sigma[\us])$ sends holomorphic sections with respect to $\tilde{K}$ to holomorphic sections with respect to $K$, in a surjective manner. 
\end{fct}

Now assume moreover that:
\begin{enumerate}
\item The involution $\theta$ of $G$ extends to an $F$-rational involution of $\tilde{G}$, still denoted $\theta$.
\item $(M,x_u)$ is a vertex (i.e., $\theta_u(M)=M$). 
\item There exist $\tilde{m}$ in $\tilde{M}$ and $\tilde{h}$ in $\tilde{H}=\tilde{G}^\theta$ in $\tilde{G}$ such that $\tilde{m}u\tilde{h}\in G$. 
\end{enumerate}

Then $(M,x_{\tilde{m}u\tilde{h}})$ is a vertex as well. 
Moreover \[\theta_{\tilde{m}u\tilde{h}}=(\theta_{u})_{\tilde{m}},\] 
\[\frak{a}_{M,\BC}^*(\theta_{\tilde{m}u\tilde{h}}),-1)=\frak{a}_{M,\BC}^*((\theta_{u})_{\tilde{m}},-1)=\frak{a}_{M,\BC}^*(\theta_u,-1),\] since 
$\tilde{M}$ acts trivially on $\frak{a}_{M,\BC}^*\subseteq \frak{a}_{\tilde{M},\BC}^*$, and 
\[M^{\theta_{\tilde{m}u\tilde{h}}}=\tilde{m} M^{\theta_u}\tilde{m}^{-1}.\] Finally, the map 
\[\ell\mapsto \ell_{\tilde{m}}:=\ell\circ \tilde{\sigma}(\tilde{m}^{-1})\] induces a linear isomorphism between 
$\Hom_{M^{\theta_u}}(\sigma,\BC)$ and $\Hom_{M^{\theta_{\tilde{m}u\tilde{h}}}}(\sigma,\BC)$. 
Here is the second main observation of this paragraph.

\begin{prop}\label{prop second orbit}
In the above situation, the local intertwining periods $J_{x_u,\sigma,\ell,\us}$ and $J_{x_{\tilde{m}u\tilde{h}},\sigma,\ell,\us}$ have a pole of the same order (in $\BZ$) at $\us=\underline{0}$. 
\end{prop}
\begin{proof}
We observe that $\tilde{h}$ normalizes $H$ and that $\tilde{m}$ normalizes $P$. Then, if $\tilde{f}_{\us}$ is a holomorphic section of $I_{\tilde{P}}^{\tilde{G}}(\tilde{\sigma}[\us])$ for $\us$ in some open cone of $\frak{a}_M^*$, we have:
\begin{equation*}
	\begin{split}
&\int_{\tilde{h}^{-1}u^{-1}\tilde{m}^{-1}P\tilde{m}u\tilde{h}\cap H\backslash H} \ell_{\tilde{m}}(\tilde{f}_{\us}(\tilde{m}u\tilde{h}h))dh\\
&=\int_{\tilde{h}^{-1}u^{-1} P u\tilde{h}\cap H\backslash H} \ell (\tilde{f}_{\us}(u\tilde{h}h))dh\\
&=\int_{u^{-1}Pu \cap H\backslash H} \ell (\tilde{f}_{\us}(uh\tilde{h}))dh\\
&=\int_{u^{-1}Pu \cap H\backslash H} \ell(\rho(\tilde{h})\tilde{f}_{\us}(uh))dh,
 	\end{split}
\end{equation*}
i.e., 
\[J_{x_{\tilde{m}u\tilde{h}},\sigma,\ell_{\tilde{m}},\us} (\tilde{f}_{\us})=J_{x_u,\sigma ,\ell,\us} (\rho(\tilde{h})\tilde{f}_{\us}).\] The result follows. 
\end{proof}

\begin{rem}
The results of this section typically apply to $\mathbf{G}$ an inner form of $\SL_n$ contained in an inner form $\widetilde{\mathbf{G}}$ of $\GL_n$. We will only use it in Section \ref{sec SL} for $\mathbf{G}=\SL_2$.  
\end{rem}

\subsection{Intertwining periods and  transitivity of parabolic induction}\label{sec compat trans}

It has been used in many special cases, that intertwining periods are compatible with transitivity of parabolic induction (see for 
example \cite[Lemma 4.4]{FLO}), and this follows from a simple integration in stages. We give a general statement of this type here, which will be used in Section \ref{sec unitary}. We suppose that $P\subseteq Q$, and that $L$ is a Levi subgroup of $Q$ contained in $M$, and we write $P=MV$ ad $Q=LU$ for the Levi decompositions. In such a situation, according to \cite[V.3.13]{Renard} there is a canonical decomposition \begin{equation}\label{eq ren} \frak{a}_{M,\BC}^*=\frak{a}_{L,\BC}^*\oplus (\frak{a}_{L,\BC}^M)^*.\end{equation} For 
$\underline{s_1} \in \frak{a}_{L,\BC}^*$ and $\underline{s_2} \in (\frak{a}_{L,\BC}^M)^*$, we denote by 
\[\Gamma_{M,L}(\sigma,\underline{s_1},\underline{s_2}):I_P^G(\sigma[\underline{s_1} +\underline{s_2}]) \simeq I_Q^G(I_{P\cap L}^L(\sigma[\underline{s_2}])[\underline{s_1}])\] the canonical isomorphism, inverse to 
\[F\mapsto F(\ )(e_L).\] Note that if  $f_{\us_1 +\us_2}$ is a holomorphic section of $I_P^G(\sigma[\us_1 +\us_2])$ and if we moreover fix $\us_2$, then $\Gamma_{M,L}(\sigma,\underline{s_1},\underline{s_2})(f_{\us_1 +\us_2})$ is a holomorphic section of $I_Q^G(I_{P\cap L}^L(\sigma[\underline{s_2}])[\underline{s_1}])$.

\begin{prop}\label{prop compat trans}
 Let $u=u_1u_2\in G$, and assume that $u_1\in L$, $\theta_u(M)=M$, and $\theta_{u_2}(L)=L$.
In such a situation we have 
\[\frak{a}_{M,\BC}^*(\theta_u,-1)=\frak{a}_{L,\BC}^*(\theta_{u_2},-1)\oplus (\frak{a}_{L,\BC}^M)^*((\theta_{u_2})_{u_1},-1),\]
and for all $\underline{s_1} \in \frak{a}_L^*(\theta_{u_2},-1)$ and $\underline{s_2} \in (\frak{a}_L^M)^*((\theta_{u_2})_{u_1},-1)$, we have 
\[J_{x_u,\sigma,\ell,\us_1+\us_2}(f_{\us_1 +\us_2})=\]
\[J_{x_{u_2}, I_{P\cap L}^L(\sigma[\underline{s_2}]),(J_{x_{u_1},\sigma,\ell,\us_1}),\us_2}\circ \Gamma_{M,L}(\sigma,\underline{s_1},\underline{s_2})(f_{\us_1 +\us_2})\] whenever $f_{\us_1 +\us_2}$ is a holomorphic section of $I_P^G(\sigma[\us_1 +\us_2])$.
\end{prop}
\begin{proof}
Note that $\theta_u=(\theta_{u_2})_{u_1}$. The first equality then follows from the fact that $L$ acts trivially on $\frak{a}_{L,\BC}^*$. Now let $T_0$ be a maximal split torus contained in $M_0$ which is $\theta_u$-stable (\cite[Lemma 2.4]{HW}). We fix a minimal parabolic subgroup of $P_0$ of $P$ containing $M_0$, so that $P$ and $Q$ are standard with respect to $(P_0,M_0)$. We can now talk about the set $R(A_M,P)\subseteq \frak{a}_M^*$ of positive roots of $A_M$. Then we set $R(A_M,L)=(\frak{a}_L^M)^*\cap R(A_M,P)$, and $R(A_L,Q)\subseteq \frak{a}_L^*$ the projection on $\frak{a}_L^*$ of $R(A_M,P)-R(A_M,L)$ with respect to the canonical decomposition \eqref{eq ren}. Associated to these positive roots $\alpha$ are coroots $\alpha^\vee$. Let $c\in \BR$ and 
\[\mathcal{D}_{M,\theta_u}^G(c):=\{\us\in \frak{a}_M^*(\theta_u,-1), \ \langle \us, \alpha^\vee \rangle>0 \ \forall \ \alpha>0, \theta_u(\alpha)<0 \}\] be a cone of convergence of $J_{x_u,\sigma,\ell,\us}$ as defined in \cite[(3.2), p.14]{MOY}. Similarly define $\mathcal{D}_{L,\theta_{u_2}}^G(c)$ and $\mathcal{D}_{M,(\theta_{u_2})_{u_1}}^L(c)$. Again because $L$ acts trivially on $\frak{a}_L^*$, we see that 
\[\mathcal{D}_{L,\theta_{u_2}}^G(c)+ \mathcal{D}_{M,(\theta_{u_2})_{u_1}}^L(c)\subseteq \mathcal{D}_{M,\theta_u}^G(c).\] We are now ready for the integration in stage argument. One has a semi-direct product decomposition $P=(P\cap L)U$. 
Now by assumption on $\theta_{u_2}$ we have $Q^{\theta_{u_2}} =L^{\theta_{u_2}}U^{\theta_{u_2}}$ thanks to Lemma \ref{lem stable decomp}. In particular
\[u^{-1}Qu\cap H=u_2^{-1}Q u_2\cap H=u_2^{-1}Q^{\theta_{u_2}} u_2 =u_2^{-1}L^{\theta_{u_2}}U^{\theta_{u_2}}u_2.\] Moreover 
\[u^{-1}Pu\cap H=u_2^{-1}(u_1^{-1}Pu_1\cap G^{\theta_{u_2}})u_2=u_2^{-1}(u_1^{-1}Pu_1\cap Q^{\theta_{u_2}})u_2.\] Now 
\[u_1^{-1}Pu_1\cap Q^{\theta_{u_2}}=u_1^{-1} (P\cap L)U u_1\cap Q^{\theta_{u_2}}=u_1^{-1}(P\cap L) u_1 U\cap L^{\theta_{u_2}}U^{\theta_{u_2}} \]
\[=(u_1^{-1}(P\cap L) u_1)^{\theta_{u_2}}U^{\theta_{u_2}}.\]
In particular 
\[u^{-1}Pu\cap H\backslash u^{-1}Qu\cap H=u_2^{-1}(u_1^{-1}(P\cap L) u_1)^{\theta_{u_2}}u_2\backslash u_2^{-1}L^{\theta_{u_2}}u_2.\]

The second equality follows from the fact that for any function 
\[f=F(\ )(e_L)\in I_P^G(\sigma[\underline{s_1} +\underline{s_2}])\] with 
$\re(\us_1+\us_2)\in \mathcal{D}_{L,\theta_{u_2}}^G(c)+ \mathcal{D}_{M,(\theta_{u_2})_{u_1}}^L(c)$, the following integration in stages of absolutely convergent integrals holds if $c$ is chosen large enough:
\begin{equation*}
	\begin{split}
&\int_{u^{-1}Pu\cap H\backslash H}\ell(f(uh))dh\\
&=\int_{u^{-1}Qu\cap H\backslash H}\int_{u^{-1}Pu\cap H\backslash u^{-1}Qu\cap H}\ell(f(u_1u_2h'h))dh'dh \\
&=\int_{u^{-1}Qu\cap H\backslash H}\int_{u_2^{-1}(u_1^{-1}(P\cap L) u_1)^{\theta_{u_2}}u_2\backslash u_2^{-1}L^{\theta_{u_2}}u_2}\ell(f(u_1u_2h))dh'dh \\
&=\int_{u^{-1}Qu\cap H\backslash H}\int_{(u_1^{-1}(P\cap L) u_1)^{\theta_{u_2}} \backslash  L^{\theta_{u_2}}}\ell(f(u_1h'u_2h))dh'dh \\
&=\int_{u^{-1}Qu\cap H\backslash H}\int_{(u_1^{-1}(P\cap L) u_1)^{\theta_{u_2}} \backslash  L^{\theta_{u_2}}}\ell(F(u_2h)(u_1h'))dh'dh 
	\end{split}
\end{equation*}
\end{proof}

This result is used in Section \ref{sec unitary} to identify the normalized intertwining periods of \cite{FLO} to intertwining periods of interest to us in this paper. In the rest of this paper, verify that Conjecture \ref{conj main} holds in several examples. 

\section{The group case}\label{sec group case}

Here we consider pairs of the form $(G,\Delta(H))$ where $G:=H\times H$ for $H$ an $F$-reductive group, and where $\Delta:h\to (h,h)$ is the diagonal embedding. We identify $H$ and $\Delta(H)$ hoping that this will not create too much confusion. Hence $\theta(x,y)=(y,x)$. We prove that Conjecture \ref{conj main} holds in this case. 

Let $P$ and $P'$ be parabolic subgroups of $G$, and $M$ and $M'$ respective Levi subgroups. We fix a maximal split torus $T_0$ of $G$ which we take of the form $T_{0,H}\times T_{0,H}$ for $T_{0,H}$ a maximal split torus of $H$, and $P_0$ a minimal parabolic subgroup of $G$ containing it, which we take of the form $P_0=P_{0,H}\times P_{0,H}$ for $P_{0,H}$ a minimal parabolic subgroup of $H$. We denote by $W_H$ the Weyl group of $H$ with respect to $T_{0,H}$. Without loss of generality, thanks to Section \ref{sec indep conj1}, we assume that $P,\ P', \ M$ and $M'$ are standard for these choices. Later when we consider holomorphic sections, they will be chosen with respect to any maximal compact subgroup $K_H\times K_H$ of $G$, where $K_H$ is in good position with respect to $(P_0,M_0)$. It then follows from the proof of Lemma \ref{lm double good position} that $K_H$ is automatically in good position with respect to $(P,M)$ and $(P',M')$. In this situation $P=P_H\times P_H'$ for 
$P_H,\ P_H'$ standard parabolic subgroups of $G$, and $M=M_H\times M_H'$ for $M_H,\ M_H'$ the standard Levi subgroups of $P_H,\ P_H'$ respectively. We put 
\[T(M_H,M_H')=\{g\in G, \ M_H'=gM_Hg^{-1}\}.\] This set is clearly stable by left translation under $M_H$, and 
\[T(M_H,M_H')/M_H=W(M_H,M_H')/W_{M_H},\] where \[W(M_H,M_H')=W_H\cap T(M_H,M_H'),\] and $W_{M_H}$ is the Weyl group of $M_H$ with respect to $S_H$. 

The symmetric space 
\[X:=\{x\in G, \ \theta(x)=x^{-1}\}\] identifies with $H$ via the map \[h\in H\to (h^{-1},h)\in X,\] and also with $G/H$ via the map 
\[gH\in G/H\to g\theta(g)^{-1}\in X.\] 
In view of these identifications, and the discussion of admissible orbits in \cite[Section 3.7]{MO}, one can check that the map 
\[w\to u:=(e,w)\] is a bijection from the set $W(M_H,M_H')/W_{M_H}$ to the set of representatives $u\in PuH$ such that $x_u$ is $M$-admissible. Here $x_u=(w^{-1},w)$. Now fix $w\in W(M_H,M_H')/ W_{M_H}$, $x_u=(w^{-1},w)$, and let \[\sigma=\tau\otimes \tau'\] be an irreducible representation of $G\times G$. Then \[M^{\theta_u}=\{(m,wmw^{-1}),\ m\in M_H\}\] and \[\Hom_{M^{\theta_u}}(\sigma,\BC)\neq \{0\}\iff \tau'\simeq w(\tau)^\vee,\] where \[w(\tau)=\tau(w^{-1} \ \cdot \ w).\] Fix a unique up to nonzero scalar $w(M_H)$-module isomorphism \[U_w: \tau'\simeq w(\tau)^\vee.\]
Up to nonzero scalar, the only nonzero $\ell\in \Hom_{M^{\theta_u}}(\sigma,\BC)$ is given by 
\[\ell(v\otimes v')=\langle v, \ U_w(v') \rangle.\]

Furthermore observe that \[\frak{a}_{M,\BC}^*=\frak{a}_{M_H,\BC}^*\times \frak{a}_{M_H',\BC}^*\] and that 
the map $\us\to (\us,-w(\us))$ is an isomorphism between $\frak{a}_{M_H,\BC}^*$ and $\frak{a}_{M,\BC}^*(\theta_u,-1)$. 
Now we have the standard intertwining operator 
\[A(w,\tau,\us): I_{P_H}^H(\tau[\us])\to I_{P'_H}^H(w(\tau)[w(\us)]).\] (See for example \cite[Section 2.6]{MOY}.) Then through the identification \[I_P^G(\sigma[\us,-w(\us)])=I_{P_H}^H(\tau[\us])\otimes I_{P'_H}^H(\tau'[-w(\us)]),\] the intertwining period $J_{x_u,\sigma,\ell,\us}$ is given by 
\[J_{x_u,\sigma,\ell,\us}(f_{\us}\otimes f'_{-w(\us)})=\int_{P_H'\backslash H} \langle A(w,\tau,\us) f_{\us}(h), U_w(f'_{-w(\us)}(h)) \rangle dh.\] 

Hence in the group case, admissible intertwining periods are described explicitly by the above formula in terms of standard intertwining operators. So Conjecture \ref{conj main} boils down to a conjecture on intertwining operators and contribution of admissible orbits, as we further explain. 

Take $P$ as in the above discussion. We suppose that $\sigma$ is square-integrable, i.e., $\tau$ and $\tau'$ are square-integrable. To prove Conjecture \ref{conj main} we may assume that \[\Hom_H(I_P^G(\sigma),\BC)\neq \{0\}.\] 
This means that 
\[\Hom_H(I_{P_H}^H(\tau),I_{P'_H}^H(\tau')^\vee)\neq \{0\}.\] Because $I_{P'_H}^H(\tau')^\vee \simeq I_{P'_H}^H(\tau'^\vee)$, we deduce by \cite[VII.2.4 and VII.2.5]{Renard} in the $p$-adic case, and \cite[Theorem 14.1]{KZ2} attributed to Langlands in the real case, the existence of $w\in W_H$ such that $M'_H\simeq w(M_H)$ and $\tau'^\vee \simeq w(\tau)$. So we may assume that $\tau'^\vee = w(\tau)$. For each Weyl element $w'$ such that $w'(\tau)\simeq w(\tau)$ (so in particular $w'(M)=w(M)$), fix $T_{w',w}:w'(\tau)\simeq w(\tau)$ as $w(M_H)$-modules. 
Conjecture \ref{conj main}\eqref{conj a} then amounts to claim that the space 
\[\Hom_H(I_{P_H}^H(\tau),I_{P'_H}^H(w(\tau)))\] can be generated by regularized intertwining operators $T_{w',w} A^*(w',\tau,\us)$ for $w'$ varying in the set of all Weyl elements such that $w'(\tau)\simeq w(\tau)$, with 
\[T_{w',w} A^*(w',\tau,\us)(f_{\us})(h):= T_{w',w}(A^*(w',\tau,\us)(f_{\us})(h)),\] and where $A^*(w',\tau,\us)$ is the regularization with respect to some generic direction of the standard intertwining operator $A(w',\tau,\us)$. However by \cite[Theorem 2.1, Properties (R2) and (R4)]{Aweighted}, and because $\tau$ is unitary, it is always possible to normalize by a meromorphic function the intertwining operators $A(w',\tau,\us)$ above into normalized operators $N(w',\tau,\us)$ which are holomorphic and unitary at $\us=\underline{0}$, and such that 
\[N(w^{-1},w(\tau),\underline{0})^{-1}N(w',\tau,\underline{0})=N(w^{-1}w',\tau,\underline{0}).\] 
Again for each Weyl element $w_0$ such that $w_0(\tau)\simeq \tau$, fix $T_{w_0}:w_0(\tau)\simeq \tau$ as $M_H$-modules (note that we can take 
$T_{w_0}=T_{w',w}$ when $w_0=w^{-1}w'$). The above discussion shows that to prove Conjecture \ref{conj main}, it is sufficient to prove that the commuting algebra \[\Hom_H(I_{P_H}^H(\tau),I_{P_H}^H(\tau))\] is generated by the self-intertwining operators $T_{w_0} N(w_0,\tau,\underline{0})$. This is a well-known theorem of Harish-Chandra: we refer to \cite[13.6]{RRG2} or the original proof in \cite[Part IV]{HC}, and \cite[Theorem 5.5.3.2]{SilHA}. 
This actually proves Conjecture \ref{conj main} in the group case (in particular Conjecture \ref{conj main}\eqref{conj b}). 

\section{The Galois pair $(\GL_n(E),\U_n(E/F))$}\label{sec unitary}

This is the most interesting and substantial example. In this section $E/F$ is a quadratic extension of $p$-adic fields with Galois involution $\tau:x\to \bar{x}$. The group $G$ is $G:=\GL_n(E)$ and $J$ is a hermitian matrix in $G$. We have the unitary involution 
\[\theta:g\to J\overline{g}^{-T}J^{-1}\] associated to $J$, and we denote by $H$ the unitary group which is its fixed point group. Here $g^{-T}$ denotes the inverse of the transpose of $g$. We set $H^o:=\GL_n(F)$. The pair $(G,H)$ is not a Gelfand pair, but the multiplicities of tempered (and more generally generic) representations of such a pair are fully understood thanks to \cite{FLO} and its sequel \cite{BPunitary}, and provide at the same time evidence and inspiration for general conjectures of Prasad (\cite{Prel}). For this pair, the papers \cite{FLO} and \cite{BPunitary} essentially provide the proof of Conjecture \ref{conj main}. However it requires some treacherous navigation between various results of \cite{FLO}, to extract the statement of Conjecture \ref{conj main} from these sources. We devote the rest of this section to a detailed explanation.

We assume that $M$ and $P$ are standard with respect to the torus of diagonal matrices contained in the Borel subgroup of upper triangular matrices, and of type for $(n_1,\dots,n_r)$, where $(n_1,\dots,n_r)$ is a composition of $n$. Let $\sigma$ be a square-integrable representation of $M$, and write it under the form $\delta_1\otimes \dots \otimes \delta_t$ where each $\delta_i$ is a square-integrable representation of $\GL_{n_i}(E)$. By \cite[Theorem 0.2]{FLO}, if $\Hom_H(I_P^G(\sigma),\BC)$ is nonzero, then $I_P^G(\sigma)$, which is irreducible, is invariant under $\tau$. Thus, up to changing $M$ by a conjugate, we may assume that 
$t=r+2s$ with $r$ and $s$ in $\BN$, $(n_1,\dots,n_t)=(n_1,\dots,n_r, m_1,\dots, m_s,m_1,\dots,m_s)$ and that 
\[\sigma\simeq \delta_1\otimes \dots \otimes \delta_r \otimes \mu_1\otimes \dots \otimes \mu_s \otimes\mu_1^\tau \otimes \dots \otimes \mu_s^\tau,\] where the representations $\delta_i$ and $\mu_j$ are square-integrable, each $\delta_i$ is fixed by $\tau$, whereas no $\mu_j$ is. By \cite[Theorem 0.2]{FLO}, the assumption that $\Hom_H(I_P^G(\sigma),\BC)$ is not reduced to zero implies that $r\geq 1$ when $H$ is not quasi-split. We recall that $``\times"$ stands for the Bernstein-Zelevinsky product notation for normalized parabolic induction. Let $BC$ be the quadratic base change map, from the set of isomorphism classes of irreducible representations of $\GL_\star(F)$ to that of isomorphism classes of irreducible representations of $\GL_\star(E)$, defined in \cite{ACBC}. Let $\omega_{E/F}$ be the quadratic character attached to $E/F$ by local class field theory. According to \cite{ACBC}, for each $i$ and $j$, there are exactly two (isomorphism classes of) square-integrable representations, say $\delta_i^o$ and $\delta_i^o\cdot\omega_{E/F}\circ\det$,  which base change to $\delta_i$, whereas there is a unique square-integrable $\mu_j^o$ such that $BC(\mu_j^0)=\mu_j\times \mu_j^\tau$. We set 
$\delta:=\delta_1\otimes \dots \otimes \delta_r,$ $\mu=\mu_1\otimes \dots \otimes \mu_s,$ and $I(\mu):=\mu_1\times \dots \times \mu_s.$ 
We define $\delta^o$, $\tau^o$ and $I(\tau^o)$ similarly. By the well-known compatibility properties of base change and parabolic induction, we have 
\[\delta  \otimes I(\mu) \times I(\mu)^\tau \simeq BC(\delta^o\otimes I(\mu^o))\] whenever $BC(\delta_i^o)=\delta_i$ 
and $BC(\mu_j^o)=\mu_j\times \mu_j^\tau$. Moreover the $2^r$ preimages $\delta^o$ of $\delta$ under $BC$ provide all the $2^r$ preimages $\delta^o \otimes I(\mu^o)$ of $\delta \otimes I(\mu) \times I(\mu)^\tau $ under $BC$. We denote by $Q$ the standard parabolic subgroup of $G$ of type $(n_1,\dots,n_r,2m_1+\dots+2m_s)$, and $L$ its standard Levi subgroup. 

In \cite[p.224]{FLO}, for each preimage $\delta^o \otimes I(\mu^o)$ of $\delta \otimes I(\mu) \times I(\mu)^\tau$, a normalizing factor $\mathfrak{n}_{L^o}(\delta^o \otimes I(\mu^o),\us)$ is defined, with $\us\in \mathfrak{a}_{L,\BC}^*$. It is a certain quotient of Shahidi's local coefficients attached to $\delta \otimes I(\mu) \times I(\mu)^\tau$ and $\delta^o \otimes I(\mu^o)$. We write $L$ under the form $\diag(M',G'')$, where $M'$ is the standard Levi subgroup of $G':=\GL_{n_1+\dots+n_r}(E)$ of type 
$(n_1,\dots,n_r)$, and $G'':=\GL_{2m_1+\dots+2m_s}(E)$. At this point, and without loss of generality, we take the matrix $J$ under the form 
\[J:=\diag(x,I_{n_1+\dots+n_r-1},I_{m_1+\dots+m_s},-I_{m_1+\dots+m_s})\] with $x\in F^\times$. In particular when $n$ is even, the group $H$ is quasi-split if and only if $x$ is a norm of $E^\times$. 
We denote by $P'$ the standard parabolic subgroup of $G'$ with standard Levi subgroup $M'$, by $P''$ the parabolic subgroup of $G''$ of type $(m_1,\dots,m_s,m_1,\dots,m_s)$, and by $M''$ its standard Levi subgroup. We put $H':=G'^{\theta'}$ and $H''=G''^{\theta''}$, where both $\theta'$ and $\theta''$ are induced by $\theta$. Note that $H''$ is quasi-split, whereas $H'$ is quasi-split if and only if $H$ is. We fix a system of representatives $R(P'\backslash G'/H')$ as in \cite[Section 6.2, before Lemma 6.4]{FLO}, and denote by $R^0(P'\backslash G'/H')$ the subset of $R(P'\backslash G'/H')$ representating open orbits.
According to \cite[Section 6.2]{FLO}, the set $R^0(P'\backslash G'/H')$ has cardinality \[|R^0(P'\backslash G'/H')|=2^{r-1}.\] Moreover for each $u'\in R^0(P'\backslash G'/H')$, the group $M'^{\theta_{u'}}$ is the product of $r$ unitary groups. According to \cite[Section 6.2]{FLO} again, there exists a unique closed double coset in $P''\backslash G''/H''$, and fixing $\iota\in E-F$ with $\iota^2\in F$, we represent it by the matrix 
\[u_1'':=\begin{pmatrix} \iota I_{m_1+\dots+m_s}& -\iota I_{m_1+\dots+m_s} \\  I_{m_1+\dots+m_s} & I_{m_1+\dots+m_s}\end{pmatrix}.\] 
Setting $u''_2:=I_{2m_1+\dots+2m_s}$, there are $2^r$ open double cosets in $Q\backslash G/H$, which can be represented by the matrices $\diag(u',u_1'')$ and $\diag(u',u_2'')$, for $u'\in  R^0(P'\backslash G'/H')$. This gives us a set $R^0(Q\backslash G/H)$ naturally partitioned into two sets $R_1^0(Q\backslash G/H)$ and $R_2^0(Q\backslash G/H)$ of cardinality $2^{r-1}$. We observe that for any 
$u\in R^0(Q\backslash G/H)$, the space $\Hom_L^{\theta_u}(\delta\otimes I(\mu)\times I(\mu)^\tau,\BC)$ has dimension equal to one thanks to 
\cite[Theorem 0.2]{FLO}, and we fix $\ell_u$ a basis of it. To each preimage $\delta^o\otimes I(\mu^o)$, and each $u\in R^0(Q\backslash G/H)$, the authors of \cite{FLO} associate in \cite[p.224]{FLO} the normalized intertwining period 
\[\mathfrak{J}_{x_u,\delta\otimes I(\mu)\times I(\mu)^\tau,\ell_u,\us}^{\delta^o\otimes I(\mu^o)}:=
\mathfrak{n}_{L^o}(\delta^o \otimes I(\mu^o),\us)J_{x_u,\delta\otimes I(\mu)\times I(\mu)^\tau,\ell_u,\us}, \ \us\in \mathfrak{a}_L^*.\] 
By \cite[Theorem 12.4, (2)]{FLO}, these $4^r$ normalized intertwining periods are all holomorphic at $\us=\underline{0}$, and we observe that at most $2^r$ of the $\mathfrak{J}_{x_u,\delta\otimes I(\mu)\times I(\mu)^\tau,\ell_u,\underline{0}}^{\delta^o\otimes I(\mu^o)}$ are linearly independent for obvious reasons (indeed for a fixed $x_u$, all linear forms $\mathfrak{J}_{x_u,\delta\otimes I(\mu)\times I(\mu)^\tau,\ell_u,\underline{0}}^{\delta^o\otimes I(\mu^o)}$ live in a one dimensional vector space). In fact $2^r$ is not optimal. Indeed,
on one hand by \cite[Theorem 12.4, (2), (12.4)]{FLO} and \cite[Proposition 13.14]{FLO}, the subspace 
of $\Hom_H(I_P^G(\sigma),\BC)$ spanned by all the linear forms $\mathfrak{J}_{x_u,\delta\otimes I(\mu)\times I(\mu)^\tau,\ell_u,\underline{0}}^{\delta^o\otimes I(\mu^o)}$ has at least dimension $2^{r-1}$. On the other hand by \cite[Theorem 3]{BPunitary}, the space $\Hom_H(I_P^G(\sigma),\BC)$ has actually dimension $2^{r-1}$. The conclusion is that $\Hom_H(I_P^G(\sigma),\BC)$ is spanned by the normalized intertwining periods $\mathfrak{J}_{x_u,\delta\otimes I(\mu)\times I(\mu)^\tau,\ell_u,\underline{0}}^{\delta^o\otimes I(\mu^o)}$. Note that we are not done yet for two reasons. 
The first one is that these normalized intertwining periods are attached to $\delta\otimes I(\mu)\times I(\mu)^\tau$ instead of $\sigma$. In order to take care of this issue, we observe that each $\ell_u$ is of the form 
\[\ell_u=\ell_{u'}\otimes \ell_{u_i''}\] for $i\in \{1,2\}$, $\ell_{u'}\in \Hom_{M'^{\theta'_{u'}}}(\delta,\BC)$ and 
 $\ell_{u_i''}\in \Hom_{G''^{(\theta'')_{u_i''}}}(I(\mu)\times I(\mu)^\tau,\BC)$. Now by \cite[Theorem 0.2]{FLO}, the space $\Hom_{G''^{(\theta'')_{u_i''}}}(\delta,\BC)$ has dimension one. When $i=1$, then $(\theta'')_{u_1''}$ stabilizes $P''$ and $\Hom_{M''^{(\theta'')_{u_1''}}}(\mu\otimes \mu^\tau,\BC)$ has dimension one with generator $\lambda_1''$. It implies that $\Hom_{G''^{(\theta'')_{u_1''}}}(I(\mu)\times I(\mu)^\tau,\BC)$ is spanned by the closed intertwining period given by the compact integration 
\[f\to \int_{P''^{(\theta'')_{u_1''}}\backslash G''^{(\theta'')_{u_1''}}}\lambda_1''(f(h''))dh''.\] Similarly, $\Hom_{M''^{(\theta'')_{u_2''}}}(\mu\otimes \mu^\tau,\BC)$ is one dimensional and it is generated by a closed intertwining period \[f\to \int_{P''^{(\theta'')_{u'' u_2''}}\backslash G''^{(\theta'')_{u''u_2''}}}\lambda_2''(f(h''))dh''\] where $u''$ is a well-chosen representative of the closed double coset 
$P''\backslash G''/G''^{(\theta'')_{u_2''}}$. We conclude from 
Proposition \ref{prop compat trans} that $\Hom_H(I_P^G(\sigma),\BC)$ is spanned by normalized intertwining periods attached to $\sigma$. 

The second problem is that so far we only normalized some intertwining periods, enough so that the normalized intertwining periods generate $\Hom_H(I_P^G(\sigma),\BC)$. As the requirement of the conjecture towards normalization is not very strong, we can normalize all the others by multiplying them by zero, although there are less trivial ways of doing this.

\begin{rem}
Here is probably the clever way to normalize all intertwining periods, which is similar to the process that we used above to normalize some of them. Because standard intertwining operators can be ``canonically" normalized, we can use 
\cite[Propositions 4.8 and 5.1]{OJNT} to reduce the normalization problem to that of normalizing 
intertwining periods of the form $J_{x_u,\sigma,\ell ,\us}$, where $(M,x_u)$ is a maximal vertex as in \cite[Definition 4.6]{MOY}. But then 
by \cite[Equality (5.9), p.26 in the proof of Theorem 5.4]{MOY}, a maximal intertwining period can be obtained by integration in stages, first as a compact integration, and then as an ``open" integration. However both open and closed intertwining periods have been normalized in 
\cite[Sections 4 and 5]{FLO}, so this provides a general process of normalization. 
\end{rem}

\section{Multiplicity one examples related to $\GL_n$ and its inner forms}\label{sec inner GL}

All examples in this section are based on the following observation, which is straightforward from the content of Section \ref{sec desc}. 

\begin{lem}\label{lm mult one}
Let $G,H,P, M, u, \ x_u$ be as in Section \ref{sec desc}, and assume that $\sigma$ is of finite length. Suppose that $\Hom_H(I_P^G(\sigma),\BC)$ is of dimension one, and that there exists a unimodular vertex $(M,x_u)$ such that $\Hom_{M^{\theta_u}}(\sigma,\BC)$ is nonzero. Then \[\Hom_H(I_P^G(\sigma),\BC)=\Hom_H^*(I_P^G(\sigma),\BC).\]
\end{lem}

Now we just observe that Conjecture \ref{conj main}\eqref{conj a} holds on known multiplicity one examples. On these examples, one should be able to extract Conjecture \ref{conj main}\eqref{conj b} from \cite{MOYlg}, but we do not try it here.

\subsection{The Galois model of $\GL_n$ and its inner forms}\label{sec gal}

In this section $E/F$ is a quadratic extension, and $D$ is an $F$-division algebra of odd dimension $d^2$ over its center $F$. This latter restriction is not important but the classification results that we use are not yet written in full generality when $d$ is even. Under our restriction on $d$, when $F$ is archimedean, then $D=F=\BR$, $E=\BC$. For $m\geq 1$ we set $G=\GL_m(D\otimes_F E)$, $H=\GL_m(D)$, and observe that $D\otimes_F E$ is a division algebra again. Such a pair being a Galois pair, it is unimodular. Moreover it is a Gelfand pair according to \cite[Proposition 11]{Fli} and more generally from \cite[Corollary A2]{FH} when $F$ is $p$-adic, and \cite[Theorem 8.2.5]{AG} when $F=\BR$: 
$\Hom_H(\pi,\BC)$ is at most one dimensional whenever $\pi$ is irreducible. We denote by $P$ the upper block triangular standard parabolic subgroup of $G$ attached to a composition $(m_1,\dots,m_r)$ of $m$, so that its standard Levi subgroup $M$ is isomorphic to 
$\GL_{m_1}(D\otimes_F E)\times \dots \times \GL_{m_r}(D\otimes_F E)$. An irreducible square-integrable representation $\sigma$ of $M$ identifies with a (completed when $F=\BR$) tensor product $\delta_1\otimes \dots \otimes \delta_r$, where each $\delta_i$ is a square-integrable representation of $\GL_{m_i}(D\otimes_F E)$. In this situation it follows from \cite[Section 3]{Matringe-GaloisModel--Generic+Classification} or more generally from \cite[Section 5.2]{MatJFA} that if $x_u$ is $M$-admissible, then 
$\Hom_{M^{\theta_u}}(\sigma,\BC)$ is nonzero if and only if there exists an involution $\epsilon$ in the symmetric group $S_r$ such that 
\[\delta_{\epsilon(i)}\simeq ({\delta_i^{\theta}})^\vee\] for all $i=1,\dots,r$, and moreover $\Hom_{\GL_{m_i}(D)}(\delta_i,\BC)$ is nonzero whenever $\epsilon(i)=i$. Observe as well that when $\sigma$ is square-integrable, the representation $I_P^G(\sigma)$ is irreducible thanks to \cite{Z} and more generally from \cite[6.1 Proposition]{Tglna}, so it affords multiplicity at most one of $H$-invariant linear forms. 

Conjecture \ref{conj main}\eqref{conj a} then follows at once from Lemma \ref{lm mult one}, \cite[Proposition 5.3]{MatJFA} when $F$ is $p$-adic, and \cite[Theorem 1.2]{Kem} when $F=\BR$. 

When $d$ is even, up to some easy verifications to be done and that we now explain, the above result will still hold. We refer to \cite{Suzuki+quaternion} when $d=2$ and $F$ is $p$-adic for the classification of distinguished representations induced from square-integrable ones, so by the same arguments as above Conjecture \ref{conj main}\eqref{conj a} holds in this case. Moreover it is clear that the method there claims a classification of distinguished representations induced from square-integrable ones for a general even $d$, as the the double cosets $P\backslash G/H$ are paramatrized by the same sets as for $d=2$ (see \cite{MatZ}). As well in the Archimedean case, the technique of \cite{ST} together with \cite[Theorem 5.4]{MOY} would again provide the classification needed for the pair $(\GL_{2n}(\BC),\GL_n(\BH))$, in view of the double coset decription provided by \cite{MatZ}.

\subsection{The twisted linear model of inner forms of $\GL_n$}

In this section $E/F$ is a quadratic extension and $D$ is an $F$-division algebra of dimension $d^2$ over its center $F$. The group $G$ is $G=\GL_m(D)$, and we assume that $md$ is even so that $E$ embeds as an $F$-subalgebra in $\CM_m(D)$. We then set $H$ to be the subgroup of elements in $G$ centralizing $E$. The pair $(G,H)$ is a Gelfand pair according to \cite{Guo} and \cite{BM}, and it is unimodular as well, as was verified in \cite{BM}. Conjecture \ref{conj main} now follows, just as in Section \ref{sec gal}, from Lemma \ref{lm mult one}, \cite[Theorems 1.3]{SuzJNT} and \cite[Theorems 1.2]{ST}.

\begin{rem}\label{rem linear model}
The paper \cite{Sign} also obtains similar Archimedean and non Archimedean classification results for certain linear models of type $(n,n)$ and $(n,n+1)$. In these cases the pair $(G,H)$ is not unimodular anymore, though it is known to be a Gelfand pair, as well as a tempered pair. For the same reasons as for twisted linear models, Conjecture \ref{conj main}\eqref{conj a} holds in this case, using the more general intertwining periods referred to in Remark \ref{rem non unimodular case}. 
\end{rem}

\section{The geometric lemma and the support of regularized open intertwining periods}\label{sec geom lemma and supp}

In this section $F$ is $p$-adic. Our notations are as in Section \ref{sec desc}, but we suppose that $P$ is standard as well as $M$ (with respect to fixed choices of a maximal split torus $T_0$ and $P_0$ a minimal parabolic subgroup of $G$ containing it). We fix $\sigma$ a finite length representation of $M$. 

\subsection{The geometric lemma}\label{sec geom lemma}

We prove in some generality some simple results well-known to experts, which will allow us to deal with small rank examples in the next section, and has proven useful in many occasions before in the literature, for instance in \cite{FLO}, \cite{MatBF1} and \cite{MatZ}. It relies on the geometric lemma of Bernstein and Zelevinsky, which provides a filtration of $I_P^G(\sigma)$ into $H$-submodules given by conditions on the support of the functions in the induced representation with respect to the double cosets $PuH$.

By \cite[Section 1.5]{BZ1}, we can order the double cosets in $P \backslash G / H$ as $\{ P u_i H\}_{i=1}^N$ such that
	\begin{align*}
		Y_i = \cup_{j = 1}^i P u_j H
	\end{align*}
	is open in $G$ for all $i = 1, \cdots,N$. Let
	\begin{align*}
		V_i = \{  \varphi \in I_P^G(\sigma) \mid \textup{Supp}(\varphi)  \subset Y_i\}.
	\end{align*}
	By convention we set $V_0=\{0\}$. 
	 By \cite[Section 3]{OJNT} and \cite[Proposition 4.1]{OJNT}, each $u_i$ for $i=1,\dots,n$ can be suitably chosen such that if $x_i:=x_{u_i}$, there exists a $\theta_{u_i}$-stable standard Levi subgroup $M_i$ of a standard parabolic subgroup $P_i\subseteq P$ of $G$ which satisfies  
	\begin{align*}
		\Hom_H(V_i / V_{i-1}, \BC)  \cong  \Hom_{M_i^{\theta_{u_i}}} (r_{M_i,M}(\sigma),\delta_{x_i}),
	\end{align*}
	where $r_{M_i,M}$ stands for the normalized Jacquet functor. In particular if $\sigma$ is cuspidal, then $M_i=M$ for all $i$. We choose the $u_i$'s as above.
	
\subsection{The support of invariant linear forms}

 In this section we make the assumption that there exists $1\leq N_0 \leq N$ a natural number such that $P$ is $\theta_{u_j}$-split with respect to $M$ for $j\leq N_0$, but does not have this property for $j>N_0$.  
 
 \begin{rem}
 \begin{enumerate}
 	\item  According to Lemma \ref{lm adm open vs split}, the representatives $u_j$ for $j\leq N_0$ are exactly those such that $Pu_j H$ is open and $x_j$ is $M$-admissible. 
 	\item On the examples that we are familiar with, if $P$ is $\theta_{u_k}$-split for one $u_k$, then $P$ is $\theta_{u_j}$-split of all $u_j$ such that $Pu_jH$ is open, but we do not know if this is to be generally expected.
 	\end{enumerate}
 \end{rem}
 
 Now we introduce the following terminology following \cite{FLO}.  
 
 \begin{df}\label{dfsupp}
 	\begin{enumerate} 
 		\item \label{dfa} We say that $P u_i H$ contributes to the distinction of $I_P^G(\sigma)$ if $\Hom_H(V_i / V_{i-1}, \BC)\neq \{0\}$. 
 		\item \label{dfb} We say that $\Hom_H(I_P^G(\sigma),\BC))$ is supported on $Y_{N_0}$ if the only   $P u_i H$ contributing to the distinction of $I_P^G(\sigma)$ are such that $i\leq N_0$. 
 		\item \label{dfc} We say that $\Hom_H(I_P^G(\sigma),\BC))$ is supported on open orbits if the only double cosets $P u H$ contributing to the distinction of $I_P^G(\sigma)$ are open in $G$. 
 		\item \label{dfd} We say that an element $L\in \Hom_H(I_P^G(\sigma),\BC))$ has  support outside $Y_{N_0}$ if $L_{|V_{N_0}}\equiv 0$.
 	\end{enumerate}
 \end{df}
 
 To justify the terminology, we observe that if $\Hom_H(I_P^G(\sigma),\BC)$ is supported on $Y_{N_0}$, then $H$-invariant linear forms on $I_P^G(\sigma)$ are determined by their restriction to sections supported on $Y_{N_0}$. We can actually be more precise. Let's start with the following key observation.  

\begin{lem}\label{lm k positif}
	Let $1\leq i_0\leq N_0$, and suppose moreover that there exists \[\ell\in  \Hom_{M^{\theta_{u_{i_0}}}}(\sigma,\BC)-\{0\}.\] If $\us_0$ is a generic vector in $\frak{a}_{M,\BC}^*(\theta_{u_{i_0}},-1)-\{0\}$ such that one can define the regularized intertwining period $J_{x_{i_0},\sigma,\ell}^{*,\us_0}$, then the integer $k(\us_0)$ used to define this regularization in Equation \eqref{eq reg} is non negative. 
\end{lem}
\begin{proof}
	This follows from the following two facts:
	\begin{itemize}
		\item for $\varphi_{i_0}$ supported on $Pu_{i_0}H$, the intertwining period $J_{x_{i_0},\sigma,\ell,\us}(\varphi_{i_0,\us})$ is defined by convergent integrals,
		\item it is nonzero for at least one choice of $\varphi_{i_0}$. 
	\end{itemize}
\end{proof}

Lemma \ref{lm k positif} has the following consequence.

\begin{prop}\label{prop support}
	Let $1\leq i_0\leq N_0$, and suppose that there exists \[\ell\in  \Hom_{M^{\theta_{u_{i_0}}}}(\sigma,\BC)-\{0\}.\] Furthermore suppose that $J_{x_{i_0},\sigma,\ell,\us}$ is not holomorphic at $\us=\underline{0}$. If $\us_0$ is a generic vector in $\frak{a}_{M,\BC}^*(\theta_{u_{i_0}},-1)-\{0\}$ such that one can define the regularized intertwining period 
	$J_{x_{i_0},\sigma,\ell}^{*,\us_0}$, then $J_{x_{i_0},\sigma,\ell}^{*,\us_0}$ is supported outside $Y_{N_0}$ (see Definition \ref{dfsupp}, \ref{dfd}).
\end{prop}
\begin{proof}
	Let $Pu_iH$ be double coset with $1\leq i_0\leq N_0$, and take $\varphi_i \in I_P^G(\sigma)$  supported on $Pu_iH$. Suppose first that $i\neq i_0$. Then by definition of the integral defining 
	$J_{x_{i_0},\sigma,\ell,\us}$, one has 
	$J_{x_{i_0},\sigma,\ell,\us}(\varphi_{i,\us})\equiv 0$ hence in particular $J_{x_{i_0},\sigma,\ell}^{*,\us_0}(\varphi_i)= 0.$ Now  as $J_{x_{i_0},\sigma,\ell,\us}$ is not holomorphic at $\us=\underline{0}$ by assumption, this forces the integer $k(\us_0)$ to be positive according to Lemma \ref{lm k positif}. But then $J_{x_{i_0},\sigma,\ell,\us}(\varphi_{i_0,\us})$ being holomorphic, this implies that $J_{x_{i_0},\sigma,\ell}^{*,\us_0}(\varphi_{i_0})= 0$.
\end{proof}
	
	 The following theorem, which generalizes \cite[(6.6) and Lemma 6.7]{FLO}, but the proof of which is the same as in [ibid.], makes the situation very precise. 
	
	\begin{thm}\label{thm: open supp}
	If $\Hom_H(I_P^G(\sigma),\BC))$ is supported on $Y_{N_0}$ (see Definition \ref{dfsupp}, \ref{dfb}), then the restriction map 
	$L\to L_{|V_{N_0}}$ provides a vector space isomorphism 
	 \[\Hom_H(I_P^G(\sigma),\BC))\simeq \Hom_H(V_{N_0},\BC).\] Moreover for $i=1,\dots,N_0$ and $\ell_i\in \Hom_{M^{\theta_{u_i}}} (\sigma,\BC)$, the open intertwining period $J_{x_i,\sigma,\ell_i,\us}$ is holomorphic at $\us=\underline{0}$, and the map \[(\ell_1,\dots,\ell_{N_0})\to \sum_{i=1}^{N_0} J_{x_i,\sigma,\ell_i,\underline{0}}\] provides an isomorphism 
	 \[\prod_{i=1}^{N_0}\Hom_{M^{\theta_{u_i}}} (\sigma,\BC)\simeq \Hom_H(I_P^G(\sigma),\BC)).\]
	\end{thm}
	\begin{proof}
		For $i=1,\dots, N_0$, let us set $
			W_i = \{  \varphi \in I_P^G(\sigma) \mid \textup{Supp}(\varphi)  \subset P u_i H\},$ so that 
		$V_{N_0}=\oplus_{i=1}^{N_0}W_i.$ The injectivity of the restriction map $L\to L_{|V_{N_0}}$ follows from the assumption that $\Hom_H(I_P^G(\sigma),\BC))$ is supported on open double cosets. Now for $i=1,\dots,N_0$ and $\ell_i\in \Hom_{M^{\theta_{u_i}}} (\sigma,\BC)$, the open period  $J_{x_i,\sigma,\ell_i,\us}$ is holomorphic at $\us=\underline{0}$  thanks to Proposition \ref{prop support}. But by an explicit form of Frobenius reciprocity, the map $\ell_i\to (J_{x_i,\sigma,\ell_i,\underline{0}})_{|W_i}$ provides an isomorphism between $\Hom_{M^{\theta_{u_i}}}(\sigma,\BC)$ and $\Hom_H(W_i,\BC)$.  Hence, because $(J_{x_i,\sigma,\ell_i,\underline{0}})_{|W_j}$ vanishes if $1\leq i\neq j\leq N_0$ as already observed in the proof of Proposition \ref{prop support}, the map  
		$(\ell_1,\dots,\ell_{N_0})\to (\sum_{i=1}^{N_0} J_{x_i,\sigma,\ell_i,\underline{0}})_{|V_{N_0}}$ is an isomorphism from $\prod_{i=1}^{N_0}\Hom_{M^{\theta_{u_i}}}(\sigma,\BC)$ to $\Hom_H(V_{N_0},\BC)\simeq\prod_{i=1}^{N_0}\Hom_H(W_i,\BC)$. The remaining claims in the statement follow from this observation.
	\end{proof}
	
	An immediate corollary of Theorem \ref{thm: open supp} is the following.
	
	\begin{cor}\label{cor open support cusp}
		Let $(G,H)$ be unimodular. Assume that $\sigma$ is cuspidal of finite length, and that $\Hom_H(I_P^G(\sigma),\BC))$ is supported on open double cosets (see Definition \ref{dfsupp}, \ref{dfc}). Then $\Hom_H(I_P^G(\sigma),\BC)=\Hom_H^*(I_P^G(\sigma),\BC)$.
	\end{cor}	
	\begin{proof}
		 The cuspidality of $\sigma$ forces that only $M$-admissible orbits contribute to distinction, so we can apply Theorem \ref{thm: open supp}.
	\end{proof}
	
	In particular. 
	
\begin{cor}\label{cor compact}
Suppose that $H/A_G\cap H$ is compact, and that $\sigma$ is cuspidal of finite length. Then $\Hom_H(I_P^G(\sigma),\BC)=\Hom_H^*(I_P^G(\sigma),\BC)$. Moreover all nonzero intertwining periods are holomorphic and nonzero at $\us=\underline{0}$. 
\end{cor}
\begin{proof}
In this situation, all $(P,H)$-double cosets are closed and hence open.
\end{proof}

\section{Some symmetric pairs with $G$ of semi-simple split rank one}\label{sec small}

In this section $F$ is always $p$-adic. In many of the situations that we study in this section, the following arguments will be used. Let $(G,H)$ be a unimodular pair such that $G$ has semi-simple rank one, and set $G'=\mathbf{G}'(F)$ the $F$-points of the derived subgroup of $\mathbf{G}$. Let $P_0$ be a proper parabolic subgroup of $G$. 

Now we consider $u\in G$. Then $\theta_u(P_0)$ which is again a proper parabolic subgroup of $G$, is equal to $P_0$, or opposite to $P_0$. In the first case $P_0uH$ is closed in $G$ according to \cite[Proposition 13.3]{HW} whereas in the second case, $P_0uH$ is open in $G$ according to \cite[Proposition 13.4]{HW}. In other words the $(P_0,H)$-double cosets are either closed or open in the rank one case, and hence the intertwining periods of \cite{MOY} are well-defined even when $(G,H)$ is not unimodular.  

Now we suppose that $P_0uH$ is open, and set $M_0:=P_0\cap \theta_u(P_0)$. Set $M_0'=(\mathbf{M_0}\cap \mathbf{G}')(F)$. We recall the canonical decomposition 
\[\frak{a}_{M_0,\BC}^*=\frak{a}_{M_0',\BC}^*\oplus \frak{a}_{A_{M_0},\BC}^*,\] where $\frak{a}_{M_0'}^*$ is of dimension one and identified with $\BC$ by choosing the weight corresponding to $P_0$. Hence for $\sigma$ an admissible representation of $M_0$, we can consider holomorphic sections $f_s\in I_{P_0}^G(\sigma[s])$ for $s\in \frak{a}_{M_0',\BC}^*$. Moreover we have \[\frak{a}_{M_0,\BC}^*(\theta_u,-1)=
\frak{a}_{M_0',\BC}^*(\theta_u,-1)\oplus \frak{a}_{A_{M_0},\BC}^*(\theta_u,-1),\] where actually $\frak{a}_{M_0',\BC}^*(\theta_u,-1)=\frak{a}_{M_0',\BC}^*$ since if not, $\theta_u$ would act as the identity on $T_0$ a $\theta_u$-stable maximal torus of $M_0$ with respect to which $M_0$ and $P_0$ are standard, and $P_0$ would be $\theta_u$-stable. If $\sigma=\chi$ is a character of $M_0$, then $\Hom_{M_0^{\theta_{x_u}}}(\chi,\BC)$ is nonzero if and only if $\chi$ is trivial on $M_0^{\theta_{x_u}}$, and we can take $\ell$ to be the identity of $\BC$ as a generator of $\Hom_{M_0^{\theta_{x_u}}}(\chi,\BC)$. We can then consider open intertwining periods of the form $J_{x_u,\chi,s}$ for $s\in \frak{a}_{M_0',\BC}^*$, where we remove $\ell$ from the notation. Now we fix $M_0$ a $\theta$-stable Levi subgroup of $P_0$. Because $M_0$ is minimal, it follows from \cite[Section 3]{OJNT} that we can find a set of representatives $R(P_0\backslash G/H)=\{u_i,i=1,\dots,N\}$, such that $
\theta_{u_i}(M_0)=M_0$ for all $i=1,\dots,N$. We fix such a choice. Moreover as before we assume that $P u_i H$ is open for $i\leq N_0$ and that it is closed for $i>N_0$. 


\begin{prop}\label{prop y}
Suppose that $\Hom_H(V_i / V_{i-1}, \BC)=\{0\}$ for $i>N_0$ except for $i=N$ (for example if $P_0u_N H$ is the only closed $(P_0,H)$-double coset). Let $\chi$ be a character of $M_0$, such that for all  $i=1,\dots,r$ with $r\leq N_0$, $\chi$ is trivial on $M_0^{\theta_{u_i}}$. If each $J_{x_{u_i},\chi,s},\ i=1,\dots,r$ has a pole of order one at $s=0$, then there exists scalars $c_1,\dots,c_r \in \BC$, at least two of them which are nonzero, such that $\sum_{i=1}^r c_i J_{x_{u_i},\chi,s} $ is regular (and automatically nonzero) at $s=0$. 
\end{prop}
\begin{proof}
 Consider the regularizations $J_{x_{u_i},\chi}^*$ at $s=0$ with respect to a fixed nonzero $s_0\in \mathfrak{a}_{M_0',\BC}^*$. According to Proposition \ref{prop support}, none of them is supported on open orbits, so by the geometric lemma, they all live in a one dimensional space, hence there exists 
 scalars $c_1,\dots,c_r \in \BC$, at least two of them which are nonzero necessarily, such that $\sum_{i=1}^r c_i J_{x_{u_i},\chi}^*=0$. These are the desired scalars from the statement. 
\end{proof} 

In the remaining sections, we will apply the above observations to specific pairs where $\mathbb{G}$ is $\SL_2$, and we will perform some explicit computations inspired from \cite[Proposition B17]{FH}, in order to prove Conjecture \ref{conj main}. 

\subsection{Pairs with $\mathbb{G}=\SL_2$}
 
\subsubsection{The Galois case} \label{sec SL}
Let $E=F[\iota]$ be a quadratic extension of $F$ with $\iota^2\in F^\times\setminus (F^\times)^2$. In this section we consider the Galois pairs $(\SL_2(E),H)$ where $H$ is either conjugate to $\SL_2(F)$ or to $\SL_1(D)$, where $D$ is a quaternionic algebra over $F$ contained in $\CM_2(E)$. 

We discuss the case where $H$ is conjugate to $\SL_2(F)$, as Conjecture \ref{conj main} in the other case follows immediatly from Corollary  \ref{cor compact}. Our involution $\theta$ is induced from the Galois conjugation $z\to \overline{z}$ of $E/F$. Denote by $\omega_{E/F}$ the quadratic character of $F^\times$ associated to $E$ by the local class field theory. Let $\N_{E/F}$ be the norm map defined by $\N_{E/F}(e)=e\bar{e}$ for $e\in E$ with kernel $E^1$. Let $P=B$ be the upper triangular Borel subgroup of $\SL_2(E)$, and $M=T$ its diagonal torus, which we identify to $E^\times$ via the map $z\to \diag(z,z^{-1})$. Hence for $\chi$ a character of $E^\times$, we set $I(\chi):=I_B^G(\chi)$.

The following theorem is extracted from \cite{ADP03}, in view of Proposition \ref{prop GL conj in SL}. 

\begin{thm}\label{ap:mult}
	Suppose that $H$ is conjugate to $\SL_2(F)$.
	\begin{enumerate}
		\item If $\chi$ is trivial, then $\dim\Hom_H(I(\chi),\mathbb{C})=2$.
		\item If $\chi=\omega_{E'/F}\circ \N_{E/F}$ where $E'$ is a quadratic field extension of $F$
		different from $E$, then $\dim\Hom_H(I(\chi),\mathbb{C})=3$.
		\item If $\chi=\chi_F\circ \N_{E/F}$ with $\chi_F^2\neq\triv$, then $\dim\Hom_H(I(\chi),\mathbb{C})=2$.
		\item If $\chi|_{F^\times}=\triv$ while $\chi^2\neq\triv$, then $\dim\Hom_H(I(\chi),\mathbb{C})=1$.
	\end{enumerate}
Otherwise $\Hom_H(I(\chi),\triv)=\{0\}$.
\end{thm}

We set $\tilde{G}=\GL_2(E)$. Thanks to Proposition \ref{prop GL conj in SL} again, in order to prove the conjecture, the choice of $H$ inside the $\tilde{G}$-conjugacy class of $\SL_2(F)$ does not matter. We choose 
\[H:=\left\{\begin{pmatrix}
	a&b\\\bar{b}&\bar{a}
\end{pmatrix}:a,b\in E\mbox{ and } \N_{E/F}(a)-\N_{E/F}(b)=1\right\}.\] This group is conjugate to the group $v_2\SL_2(F) v_2^{-1}$ where 
\[v_2:=\begin{pmatrix} 1 & -\iota \\ 1 & \iota \end{pmatrix}\in \tilde{G}.\] Actually we denote $\tilde{H}$ the subroup of $\tilde{G}$ given by the same matrices, but without restriction on the determinant, which is in fact $v_2\GL_2(F)v_2^{-1}$. We also denote by $\tilde{B}$ the Borel subgroup of upper triangular matrices, and by $\tilde{T}$ its diagonal torus. 

From \cite[Page 488]{lu2018pacific} and Section \ref{sec indep conj2}, there are three $(B,H)$-double cosets in $G$. The double coset decomposition is 
\[G=Bu_0H\sqcup B u_1 H \sqcup B u_2 H,\] where \[u_2=\diag(2\iota, 1 ) v_2^{-1}\] is such that $B u_2 H$ is the unique closed double coset, 
\[u_0=I_2,\]  and \[u_1=\diag(1,\epsilon) v_1\] where $\epsilon$ is an element in $F^\times\setminus \N_{E/F}(E^\times)$ and $v_1$ is a matrix $\begin{pmatrix}
	a&b\\\bar{b}&\bar{a}
\end{pmatrix}\in \tilde{G}$ such that $\N_{E/F}(a)-\N_{E/F}(b)=\epsilon^{-1}$. We observe that 
\[\tilde{G}=\tilde{B}u_0 \tilde{H} \sqcup \tilde{B} u_2 \tilde{H}= \tilde{B}u_1 \tilde{H} \sqcup \tilde{B} u_2 \tilde{H}.\]
Actually 
\[\tilde{B}u_0 \tilde{H}\cap G=\tilde{B}u_1 \tilde{H}\cap G=Bu_0H\sqcup B u_1 H \] and \[G\cap \tilde{B} u_2 \tilde{H}=B u_2 H .\]

We set $\tilde{H}^+$ to be the index $2$ subgroup of $\tilde{H}$ given by matrices with determinant in $\N_{E/F}(E^\times)$.

Note that 
\begin{equation} \label{eq dec a} \tilde{B} u_0 \tilde{H}=\tilde{B} u_0\tilde{H}^+\sqcup \tilde{B} u_1 \tilde{H}^+,\end{equation} 
\begin{equation} \label{eq dec b} u_0^{-1}\tilde{B} u_0\cap \tilde{H}^+\backslash \tilde{H}^+=u_0^{-1}B u_0\cap H\backslash H,\end{equation} 
\begin{equation} \label{eq dec c} u_1^{-1}\tilde{B} u_1\cap \tilde{H}^+\backslash \tilde{H}^+=u_1^{-1}B u_1\cap H \backslash H,\end{equation} 
and 
\begin{equation} \label{eq dec d} u_2^{-1}\tilde{B} u_2\cap \tilde{H} \backslash \tilde{H} =u_2^{-1}B u_2\cap H \backslash H,\end{equation} 

In the rest of this section, we assune that $\chi$ is of the form \[\chi=\eta\circ \N_{E/F}\] for $\eta$ a unitary character of $F^\times$, and we denote by $\eta$ again an extension of $\eta$ to $E^\times$ (which is unitary necessarily). We then define the character 
\[\tilde{\chi}:=\eta\otimes \overline{\eta}^{-1}\] of $\tilde{T}$, so that restriction of functions from 
$I(\tilde{\chi}[s]):=I_{\tilde{B}}^{\tilde{G}}(\tilde{\chi})$ to $G$ is a $G$-module isomorphism between $I(\tilde{\chi}[s])$ and $I(\chi[s])$. We moreover choose $K=\SL_2(\mathcal{O}_E)$ as a maximal compact subgroup of $G$, and $\tilde{K}$  as a maximal compact subgroup of $\tilde{G}$, and take flat sections of $I(\tilde{\chi}[s])$ with respect to $\tilde{K}$, so that their restriction to $G$ are flat sections of $I(\chi[s])$ with respect to $K$.  

Recall that $I(\tilde{\chi})$ is always distinguished by $\tilde{H}$, and that $\tilde{B} u_1 \tilde{H}$ contributes to its distinction if and only if $\eta$ is trivial on $F^\times$, which in particular implies that $\chi$ is trivial. In particular, from Theorem  \ref{thm: open supp}, we know that if $\eta$ is not trivial on $F^\times$, then $J_{x_{u_0},\tilde{\chi},s}$ is holomorphic at $s=0$. On the other hand, if $\eta$ is trivial on $F^\times$, it follows from \cite[Proposition 10.9]{MatJFA} that $J_{x_{u_0},\tilde{\chi},s}$ has a pole at $s=0$, and then from Equation \eqref{eq rel int conj} and \cite[Proposition 4.5]{MatZ} (see \cite[Lemma 27]{JLR} when $E/F$ is unramified) that this pole is of order one: indeed one can always majorize any holomorphic section of $I(\tilde{\chi}[s])$ by a positive multiple of the spherical section. 

\begin{concl*}
The open intertwining period $J_{x_{u_0},\tilde{\chi},s}$ on $I(\tilde{\chi}[s])$ is regular at $s=0$ except when $\eta$ is trivial on $F^\times$, in which case it has a pole of order one. 
\end{concl*}

Now it follows from Equations 
\eqref{eq dec a}, \eqref{eq dec b} and \eqref{eq dec c} that 

\begin{equation} \label{eq sum int} J_{x_{u_0},\tilde{\chi},s}(f_s)=J_{x_{u_0},\chi,s}(f_s)+J_{x_{u_1},\chi,s}(f_s). \end{equation}

Suppose that $J_{x_{u_0},\chi,s}(f_s)$ has a pole at $s=0$. Then its regularization at $s=0$ is supported on no open orbit, and depends only on $f_{|Bu_2H}$. Hence we may assume that $f_s$ is supported on $Bu_0H\sqcup Bu_2H$. Let $f_s$ be the holomorphic section of 
$I(\tilde{\chi}[s])$ which restricts to $f_s\in I(\chi[s])$. Then by Equation \eqref{eq sum int}, we have 
\[J_{x_{u_0},\tilde{\chi},s}(f_s)=J_{x_{u_0},\chi,s}(f_s).\]  This tells us two things:
\begin{enumerate}
\item if $J_{x_{u_0},\chi,s}$ has a pole at $s=0$, then $\chi$ is trivial;
\item moreover the order of the pole of $J_{x_{u_0},\chi,s}$ that has a pole at $s=0$ is at most equal to one. 
\end{enumerate}

 Actually we can claim the same for $J_{x_{u_1},\chi,s}$ thanks to Proposition \ref{prop second orbit} (or by the above argument). Conversely, if $\chi$ is trivial, we may always assume that $\eta$ is trivial on $F^\times$, and it follows from  Equation \eqref{eq sum int} again that either $J_{x_{u_0},\chi,s}$  or $J_{x_{u_1},\chi,s}$ has a pole at $s=0$, hence from Proposition \ref{prop second orbit} that they both do.

\begin{concl*}
Suppose that $\chi$ is trivial on $E^1=\ker \N_{E/F}$. For $i=0,1$, the open intertwining period $J_{x_{u_i},\chi,s}$ on $I(\chi[s])$ is regular at $s=0$ except when $\chi=\triv$, where it has a simple pole of order one. 
\end{concl*}

We are now ready to prove the following. 

\begin{thm}\label{SL_2}
Let $\chi$ be a character of $E^\times$. Then $\Hom_H(I(\chi),\BC)=\Hom_H^*(I(\chi),\BC)$ for $G=\SL_2(E)$ and $H=\SL_2(F)$ or $\SL_1(D)$. Moreover Conjecture \ref{conj main}\eqref{conj b} also  holds for these pairs. 
\end{thm}
\begin{proof}
The statement when $H$ is conjugate to $\SL_1(D)$ follows at once from Corollary \ref{cor compact}. Now we suppose that $H$ is conjugate to $\SL_2(F)$. 
	\begin{enumerate}[(a)]
		\item If $\chi$ is trivial, then both $J_{x_{u_0},\chi,s}$ and $J_{x_{u_1},\chi,s}$ have a pole of order one at $s=0$. Hence by 
		Proposition \ref{prop y}, there exists $(c_0,c_1)\in \BC^2-\{0\}$ such that $c_0J_{x_{u_0},\chi,s}+c_1 J_{x_{u_1},\chi,s}$ is regular at $s=0$. Furthermore,
		$c_0J_{x_{u_0},\chi,0}+c_1 J_{x_{u_1},\chi,0}$  and $J_{x_{u_2},\chi,0}$ are linearly independent since $J_{x_{u_2},\chi,0}$ vanishes on $W_0+W_1$, where $W_i$ is defined as in the proof of Theorem \ref{thm: open supp}, whereas $J_{x_{u_i},\chi,0}$ vanishes on $W_j$ if $\{i,j\}=\{0,1\}$.
		 In this case $\Hom_{\SL_2(F)}(I(\chi),\mathbb{C})$ is generated by $c_0J_{x_{u_0},\chi,0}+c_1 J_{x_{u_1},\chi,0}$ and $J_{x_{u_2},\chi,0}$. Here the normalizing factors can be taken equal to $1$.
		\item If $\chi=\omega_{E'/F}\circ \N_{E/F}$, then $J_{x_{u_0},\chi,s}$ and $J_{x_{u_1},\chi,s}$ are holomorphic at $s=0$. Hence $\Hom^\ast_{\SL_2(F)}(I(\chi),\mathbb{C})$ is generated by $J_{x_{u_0},\chi,0}, \ J_{x_{u_1},\chi,0}$ and $J_{x_{u_2},\chi,0}$. Here the normalizing factors can be all taken equal to $1$ again.
		\item If $\chi=\chi_F\circ \N_{E/F}$ with $\chi_F^2
		\neq\mathbb{C}$, then $\Hom^\ast_{\SL_2(F)}(I(\chi),\mathbb{C})$ is generated by $J_{x_{u_0},\chi,0}$ and $\ J_{x_{u_1},\chi,0}$  either by the above discussion or by Corollary \ref{cor open support cusp}. Here the normalizing factors can be all taken equal to $1$ again.
		\item If $\chi|_{F^\times}=\mathbb{C}$ while $\chi^2\neq\mathbb{C}$, then
		$\Hom^\ast_{\SL_2(F)}(I(\chi),\mathbb{C})$ is generated by $J_{x_{u_2},\chi,0}$, and $1$ is an appropriate choice of normalizing factor.
	\end{enumerate}
	The result now follows from Theorem \ref{ap:mult}.
\end{proof}

\begin{rem}
There should be more meaningful choices of normalizing factors above.
\end{rem}

\subsubsection{The linear and twisted linear model} Let $G:=\SL_2(F)$, $B$ be the upper triangular Borel subgroup of $\SL_2(F)$, and $T$ the diagonal torus of $\SL_2(F)$. Let $E$ be quadratic extension of $F$ embedded as an $F$-subalgebra of $\CM_2(F)$. We recall that $E^1:=\{x\in E^\times,\N_{E/F}(x)=1\}$. We consider the tempered pairs $(G,H)$ with $H=T$ or $H=E^1$. We recall preliminary facts for $H=T$. 

\begin{itemize} 
\item There are two closed double cosets: $BuT=B$ and $Bu'T $, where $u:=I_2$ and $u':=\begin{pmatrix}
	0&1\\-1&0
\end{pmatrix} $. 
\item There are $|F^\times/(F^\times)^2|$ open double cosets $B u_{\epsilon} T$ where $u_{\epsilon}:=\begin{pmatrix}
	1&0\\ \epsilon&1
\end{pmatrix}$, for $\epsilon$ in a system of representatives $R(F^\times/(F^\times)^2)$. 
\end{itemize}

Let $\eta$ be a character of $T$ identified with a character of $F^\times$ as before. The main theorem of this section is the following. 

\begin{thm}\label{thm sl linear}
Conjecture \ref{conj main} holds for the pair $(G,H)$. 
Moreover when $H=T$, the space $\Hom_T(I(\eta),\BC)$ is nonzero if an only if $\eta(-1)=1$, in which case we have:
\begin{enumerate}
\item if $\eta=|\ |_F^{\pm 1}$, then $ \dim(\Hom_T(I(\eta),\BC))=|F^\times/(F^\times)^2|+1$;
\item if $\eta\neq |\ |_F^{\pm 1}$, then $\dim(\Hom_T(I(\eta),\BC))=|F^\times/(F^\times)^2|$. 
\end{enumerate}
\end{thm}
\begin{proof}
The statement when $H=E^1$ follows at once from Corollary \ref{cor compact}. Now we suppose $H=T$. The necessary and sufficient condition for distinction comes from the geometric lemma and we skip the very standard computations, together with the double coset decomposition of $G$. Now we explain the multiplicities, proving that $\Hom_T(I(\eta),\BC)=\Hom_T^*(I(\eta),\BC)$ at the same time. We assume that $\eta(-1)=1$.  

\begin{itemize}
   \item All open double cosets $Bu_{\epsilon}H$ always contribute to distinction. 
	\item Moreover if $\eta\neq |\ |_F^{\pm 1}$, the space $\Hom_T(I(\eta),\BC)$ is actually supported on open orbits, hence 
	\[\dim(\Hom_T(I(\eta),\BC))=|F^\times/(F^\times)^2|\] and the linear forms $J_{x_u,\eta,\underline{0}}$ form a basis of $\Hom_T(I(\eta),\mathbb{C})$. This already proves Conjecture \ref{conj main}, as the conjecture is for unitary $\eta$.
	\item Before treating the cases where $\eta=|\ |_F^{\pm 1}$, let us express the open intertwining periods in terms of Tate integrals. 
	For $f_s$ a holomorphic section of $I(\eta|\ |_F^s)$, by definition and for $\re(s)$ large enough: 
\[J_{x_{u_{\epsilon}},\eta,s}(f_s)=\int_{\mu_2\backslash T}f_s \begin{pmatrix}
	a &0\\ a \epsilon & a^{-1}
\end{pmatrix} d^\times a,\]
where $\mu_2$ is the center of $\SL_2(F)$. Observe that whenerver $a\neq 0$: \[\begin{pmatrix}
	a&0\\ a \epsilon &a^{-1}
\end{pmatrix}=\begin{pmatrix}
	a^{-1} \epsilon^{-1} & a \\0&  a \epsilon
\end{pmatrix}\begin{pmatrix}
	0&-1\\1& a^{-2}\epsilon^{-1} 
\end{pmatrix}.\]
Therefore 
\begin{equation*}
	\begin{split}
&J_{x_{u_{\epsilon}},\eta,s}(f_s)\\
&=\int_{|a^2\epsilon|_F\leq 1}f_s\left(\begin{pmatrix}
			a&0\\0&a^{-1}
		\end{pmatrix}\begin{pmatrix}
			1&0\\a^2\epsilon&1
		\end{pmatrix}\right) d^\times a +\int_{|a^2\epsilon|_F> 1}f_s \left(\begin{pmatrix}
	a^{-1} \epsilon^{-1} & a \\0&  a \epsilon
\end{pmatrix}\begin{pmatrix}
	0&-1\\1& a^{-2}\epsilon^{-1} 
\end{pmatrix} \right) d^\times a\\
&
=\int_{|a^2\epsilon|_F\leq 1}\eta(a)|a|_F^{s+1}f\begin{pmatrix}
			1&0\\a^2\epsilon&1
		\end{pmatrix} d^\times a + \eta(\epsilon)^{-1}|\epsilon|_F^{-s-1} \int_{|a^2\epsilon^{-1}|_F\leq1} \eta(a)|a|_F^{s-1} f\begin{pmatrix}
	0&-1\\1& a^2\epsilon^{-1} 
\end{pmatrix} d^\times a.
\end{split}\end{equation*}		
		We recognize the sum of two Tate integrals. The first one is holomorphic at $s=0$ except if $\eta=|\ |_F^{-1}$, in which case it has a pole of order at most one, which is realized by the spherical section. The second one is holomorphic at $s=0$ except if $\eta=|\ |_F$, in which case it has a pole of order at most one, which is also realized by the spherical section. 

	\item If $\eta=|\ |_F^{\pm 1}$, then all orbits contribute to the distinction of $I(\sigma)$. Now write $R(F^\times/(F^\times)^2)=\{\epsilon_1,\dots,\epsilon_r\}$. By the above discussion, each open intertwining period $J_{x_{u_{\epsilon_i}},\eta,s}$ has a pole of order one at $s=0$, hence by Proposition \ref{prop y} there exist nonzero scalars $c_1,\dots,c_r$ such that $c_{i+1}J_{x_{u_{\epsilon_{i+1}}},\eta,s}-c_iJ_{x_{u_{\epsilon_{i}}},\eta,s}$ is holomorphic at $s=0$ for $i=1,\dots,r-1$. But then, exactly as in the proof of (a) of Theorem \ref{SL_2}, the family 
	\[((c_{i+1}J_{x_{u_{\epsilon_{i+1}}},\eta,0}-c_iJ_{x_{u_{\epsilon_i}},\eta,0})_{i=1,\dots,r}, J_{x_u,\chi,0},J_{x_{u'},\chi,0})\] is linearly independent, hence \[\dim(\Hom_T(I(\eta),\BC))\geq |F^\times/(F^\times)^2|+1.\] It is now sufficient to prove that \[\dim(\Hom_T(I(\eta),\BC))\leq |F^\times/(F^\times)^2|+1.\] However $I(\eta)$ has length two, with composition factors the Steinberg $\mathrm{St}$ representation and the trivial representation. Hence it is sufficient to prove that \[\dim(\Hom_T(\mathrm{St},\BC))\leq |F^\times/(F^\times)^2|.\] We denote by 
	$\tilde{\mathrm{St}}$ the Steinberg representation of $\GL_2(F)$, and recall that its restriction to $\SL_2(F)$ is just $\mathrm{St}$. 
	Now, following \cite{ADP03} in the Galois case, we claim that $\Hom_T(\mathrm{St},\BC)$ is an $F^\times/(F^\times)^2$-module, where $\overline{t}\cdot L:=L\circ \diag(t,1)$, so it decomposes into $|F^\times/(F^\times)^2|$ weight spaces. However, denoting by $\tilde{T}$ the diagonal torus of $\GL_2(F)$, it is well-known that $\dim \Hom_{\tilde{T}}(\tilde{\mathrm{St}},\chi\otimes \chi^{-1})=1$ whenever $\chi$ is a quadratic character of $F^\times$. This implies that $\Hom_T(\mathrm{St},\BC)$ is the direct sum of $|F^\times/(F^\times)^2|$ one dimensional weight spaces, hence the result. 
	\end{itemize}
	\end{proof}

\subsection*{Acknowledgement.} We thank Patrick Delorme, Yangyu Fan and Yiannis Sakellaridis for directing us towards this question. We thank Erez Lapid for useful conversations, and Miyu Suzuki for useful comments on a previous version of this draft. Most importantly, we thank the anonymous referee for their very accurate  comments and
corrections, which lead to needed clarifications in several key places. The first named author was partially supported by NSFC 12301031.

	\bibliographystyle{alphanum}
	\bibliography{references}
	
\end{document}